\documentclass{amsart}
\usepackage{amssymb}
\mathchardef\mhyphen="2D
\usepackage{comment}
\usepackage{xy}
\usepackage{mathtools}
\xyoption{all}
\usepackage[T1]{fontenc}
\usepackage{hyperref}
\usepackage{cleveref}
\DeclareMathOperator{\Comod}{{\rm Comod}}
\DeclareMathOperator{\holim}{{\rm holim}}
\DeclareMathOperator{\Aut}{{\rm Aut}}
\DeclareMathOperator{\cod}{{\rm cod}}
\DeclareMathOperator{\op}{{\rm op}}
\DeclareMathOperator{\colim}{{\rm colim}}
\DeclareMathOperator{\Ext}{{\rm Ext}}
\DeclareMathOperator{\conn}{{\rm conn}}
\DeclareMathOperator{\coker}{{\rm coker}}
\DeclareMathOperator{\Spec}{{\rm Spec}}
\DeclareMathOperator{\Sq}{{\rm Sq}}
\DeclareMathOperator{\ann}{{\rm ann}}
\DeclareMathOperator{\idealsets}{{\rm idealsets}}
\DeclareMathOperator{\Ab}{{\rm Ab}}
\DeclareMathOperator{\homog}{{\rm homog}}
\DeclareMathOperator{\grad}{{\rm grad}}
\DeclareMathOperator{\Mit}{{\rm Mit}}
\DeclareMathOperator{\tr}{{\rm tr}}
\DeclareMathOperator{\id}{{\rm id}}
\DeclareMathOperator{\ob}{{\rm ob\:}}
\DeclareMathOperator{\comod}{{\rm comod}}
\DeclareMathOperator{\gr}{{\rm gr}}
\DeclareMathOperator{\Mod}{{\rm Mod}}
\DeclareMathOperator{\dist}{{\rm dist}}
\theoremstyle{plain}
\newtheorem{prop}{Proposition}[section]
\newtheorem{theorem}[prop]{Theorem}
\newtheorem{corollary}[prop]{Corollary}
\newtheorem{setup for thm}[prop]{Setup for theorem}
\newtheorem{necessary condition}[prop]{Necessary condition}

\newtheorem*{unnumberedtheorem}{Theorem}
\newtheorem*{unn motivating questions}{Motivating Questions}
\newtheorem{lemma}[prop]{Lemma}

\newtheorem{motivating questions}[prop]{Motivating Questions}
\newtheorem{definition}[prop]{Definition}
\newtheorem{definition-proposition}[prop]{Definition-Proposition}
\newtheorem{definition-theorem}[prop]{Definition-Theorem}

\newtheorem{running notations}[prop]{Running notations}

\theoremstyle{definition}

\newtheorem{remark}[prop]{Remark}
\newtheorem{intuitive idea}[prop]{Intuitive idea}
\newtheorem*{history-and-status-of-this-paper}{History and status of this paper}

\newtheorem{question}[prop]{Question}

\newtheorem{example}[prop]{Example}
\newtheorem{examples}[prop]{Examples}

\newtheorem{observation}[prop]{Observation}

\newtheorem{conventions}[prop]{Conventions}

\title[Derived limits in dual Steenrod algebra comodules]{Derived functors of product and limit in the category of comodules over the dual Steenrod algebra.}
\author{A. Salch}
\begin{document}
\begin{abstract}
In the 2000s, Sadofsky constructed a spectral sequence which converges to the mod $p$ homology groups of a homotopy limit of a sequence of spectra. The input for this spectral sequence is the derived functors of sequential limit in the category of graded comodules over the dual Steenrod algebra. Since then, there has not been an identification of those derived functors in more familiar or computable terms. Consequently there have been no calculations using Sadofsky's spectral sequence except in cases where these derived functors are trivial in positive cohomological degrees. 

In this paper, we prove that the input for the Sadofsky spectral sequence is the graded local cohomology of the Steenrod algebra, taken with appropriate (quite computable) coefficients. 
This turns out to require both some formal results, like some general results on torsion theories and local cohomology of noncommutative non-Noetherian rings, and some decidedly non-formal results, like a 1985 theorem of Steve Mitchell on some very specific duality properties of the Steenrod algebra not shared by most finite-type Hopf algebras. Along the way there are a few results of independent interest, such as an identification of the category of graded $A_*$-comodules with the full subcategory of graded $A$-modules which are torsion in an appropriate sense.
\end{abstract}
\maketitle
\tableofcontents

\section{Introduction.}

Let $k$ be a field.
It is well-known that the category $\Comod(\Gamma)$ of comodules over a $k$-coalgebra $\Gamma$ is abelian and has generally agreeable properties, like being complete and co-complete and having enough injectives, but unlike the category of modules over a ring, $\Comod(\Gamma)$ does not necessarily satisfy Grothendieck's axiom $AB4^*$. Recall from \cite{MR0102537} that $AB4^*$ is the axiom that states that infinite products of exact sequences are exact. Consequently, given a set $\{ M_i: i\in I\}$ of $\Gamma$-comodules, we may have nonvanishing higher right-derived functors $R^n\prod_{i\in I}^{\Gamma} M_i$ of the categorical product $\prod^{\Gamma}$ in the category of $\Gamma$-comodules. One consequence is that, given a sequence $\dots \rightarrow M_2 \rightarrow M_1 \rightarrow M_0$ of $\Gamma$-comodule homomorphisms, we may have nonvanishing right-derived functors $R^n\lim^{\Gamma}_i M_i$ for $n>1$, unlike the familiar situation in categories of modules; see \cite{MR132091} and \cite{MR2197371} for this implication.

There are important topological consequences. Let $p$ be a prime number, suppose that the ground field $k$ is the field $\mathbb{F}_p$ with $p$ elements, and suppose that the coalgebra $\Gamma$ is the $p$-primary dual Steenrod algebra. In the influential but unpublished 2001 preprint \cite{sadofsky2001homology}, H. Sadofsky constructed a spectral sequence 
\begin{align}
\label{sadofsky ss 1} E_2^{*,*} \cong R^*\lim^{\Gamma}_i H_*(X_i; \mathbb{F}_p) &\Rightarrow H_*\left( \holim_i X_i; \mathbb{F}_p\right)\end{align}
for a sequence $\dots \rightarrow X_2 \rightarrow X_1 \rightarrow X_0$ of $H\mathbb{F}_p$-nilpotently complete spectra. Since then, Sadofsky's spectral sequence has been written about, and the details of its construction are available in the published literature: for example, see section 11.3 of \cite{barthel_pstragowski_2023}, or the appendix of \cite{MR4080481}, or \cite{MR2337861} for the case of products of spectra, where one has a spectral sequence
\begin{align}
\label{sadofsky ss 2} E_2^{*,*} \cong R^*\prod^{\Gamma}_i H_*(X_i; \mathbb{F}_p) &\Rightarrow H_*\left( \prod_i X_i; \mathbb{F}_p\right)\end{align}
for any set $\{ X_i: i\in I\}$ of spectra.
We will refer to \eqref{sadofsky ss 1} as the {\em Sadofsky spectral sequence,} and \eqref{sadofsky ss 2} as the {\em Hovey-Sadofsky spectral sequence.}
Countable products can of course be treated as a special case of sequential limits, so it is easy to see the countable case of the Hovey-Sadofsky spectral sequence as a special case of the Sadofsky spectral sequence.

It is a classical theorem of Adams (Theorem III.15.2 of \cite{MR1324104}) that mod $p$ homology commutes with homotopy limits of {\em uniformly bounded-below} sequences of spectra. That is, $H_*\left( \holim_i X_i;\mathbb{F}_p\right) \rightarrow \lim_i H_*(X_i;\mathbb{F}_p)$ is an isomorphism if there is a {\em uniform} lower bound on the degrees of the nonvanishing homotopy groups of the spectra $X_0, X_1, X_2, \dots$. Sadofsky's spectral sequence is the only available general tool for calculating homology of homotopy limits of spectra in the absence of a uniform lower bound on their homotopy groups. 

However, to date there has been little known about the input of the Sadofsky or Hovey-Sadofsky spectral sequences, because it has been unclear whether there could be some practical means of calculating the derived functors of sequential limit, or of product, in categories of graded comodules\footnote{However, there are some known tools for calculating the input for special cases of generalizations of the Sadofsky spectral sequence. Under reasonable conditions (see \cite{MR2337861} or \cite{barthel_pstragowski_2023}), one can build a version of the Sadofsky spectral sequence for the Morava $E$-theory $E(\mathbb{G})_*$ of a formal group law $\mathbb{G}$, rather than mod $p$ homology. Its input is the derived limit $R^*\lim^{E(\mathbb{G})_*E(\mathbb{G})}_i E(\mathbb{G})_*(X_i)$ in the category of graded $E(\mathbb{G})_*E(\mathbb{G})$-comodules, and it converges to the $E(\mathbb{G})$-homology groups of the homotopy limit $\holim_i X_i$. In one case, there is a known description of such a derived limit in $E(\mathbb{G})_*E(\mathbb{G})$-comodules: the derived limit group $R^s\lim^{E(\mathbb{G})_*E(\mathbb{G})}_i E(\mathbb{G})_*/\mathfrak{m}_i$ is proven to be isomorphic to the continuous cohomology $H^s_{cts}(\Aut(\mathbb{G}); E(\mathbb{G})_*E(\mathbb{G}))$ of the Morava stabilizer group, in Theorem D of \cite{barthel_pstragowski_2023}. 

Along similar lines, the paper \cite{MR4060488} contains a sustained investigation of the properties of adic completion and derived adic completion on categories of comodules. These comodule completions involve limits of the form $\lim^{\Gamma}_i (M\otimes_A A/I^i)$ in the category of comodules over a Hopf algebroid $(A,\Gamma)$. 
That paper yields tools which are useful for the study of adic completions of comodules, but not the tools for study and calculation of {\em general} limits in comodule categories which are developed in this paper.}. As Behrens and Rezk write in \cite{MR4094969} about the Sadofsky spectral sequence, ``[t]he $E_2$-term of this spectral sequence is in general quite mysterious'', and as Hovey writes in \cite{MR2337861} about the Hovey-Sadofsky spectral sequence, ``[a]lmost nothing is known about these right derived functors''. 
The purpose of this paper is to develop some general theory and some practical tools for calculating derived functors of sequential limits and of products in the category of graded comodules over a graded coalgebra, with a particular emphasis on the dual Steenrod algebras as our motivating examples. The basic strategy is to use the {\em covariant} embedding of the category of graded $\Gamma$-comodules into the category of graded $\Gamma^*$-modules, in order to identify the derived functors of product in the former category with some more familiar cohomological invariant in the latter category.

Below, we survey the results in this paper, but first we present the most topologically compelling results, which appear as Corollaries \ref{derived products are local cohomology} and \ref{steenrod alg seq cor}. The notation is as follows: $p$ is any prime number, $\Gamma^*$ denotes the mod $p$ Steenrod algebra, $\Gamma$ is its dual coalgebra, $I_j$ is the ideal of the Steenrod algebra $\Gamma^*$ generated by all homogeneous elements of degree $\geq j$, and for a graded $\Gamma$-comodule $M$, we write $\iota M$ for $M$ regarded as a graded $\Gamma^*$-module via the adjoint $\Gamma^*$-action. 
\begin{unnumberedtheorem}
Let $\{ M_i: i\in I\}$ be a set of bounded-above\footnote{To be clear, the grading on all modules and comodules in this paper is the {\em cohomological} grading. Hence, if $M_i$ is the homology of a spectrum---the most relevant case for the Sadofsky and Hovey-Sadofsky spectral sequence---then the relevant assumption is that the spectrum should be bounded {\em below.}

Furthermore, the assumption here is only that {\em each} graded $\Gamma$-comodule is bounded above. There does not need to be a {\em uniform} upper bound on these comodules. In the case where there is a uniform upper bound on the comodules, none of the work we do in this paper is necessary, because it is easy to show that the comodule product is exact.} graded $\Gamma$-comodules. Then the $n$th derived functor $R^n\prod_{i\in I}^{\Gamma} M_i$ of the product of the $M_i$ in the category of graded $\Gamma$-comodules is isomorphic, as an abelian group, to \[ \underset{j\rightarrow\infty}{\colim} \Ext_{\Gamma^*}^n\left( \Gamma^*/I_j,\prod_{i\in I} \iota M_i\right).\] That is, the derived functors of product in the category of graded $\Gamma$-comodules are given, on families of bounded-above comodules, by the graded local cohomology of the Steenrod algebra $\Gamma^*$ with coefficients in the product of the adjoint $\Gamma^*$-modules.
\end{unnumberedtheorem}
\begin{unnumberedtheorem}
Let \begin{equation*}
\dots\rightarrow M_2\rightarrow M_1\rightarrow M_0\end{equation*} be a sequence of bounded-above graded $\Gamma$-comodules such that $R^1\lim$ vanishes on the sequence of graded $\Gamma^*$-modules 
$\dots\rightarrow \iota M_2\rightarrow \iota M_1\rightarrow \iota M_0$.
Then we have an isomorphism of graded $\Gamma^*$-modules
\[ \iota R^*\lim^{\Gamma}_i M_i  \cong \underset{j\rightarrow\infty}{\colim} \Ext^*_{\Gamma^*}\left(\Gamma^*/I_j,\lim_i \iota M_i\right)\]
\end{unnumberedtheorem}
That is, if $R^1\lim$ vanishes on the sequence of adjoint $\Gamma^*$-modules, then the derived functors of sequential limit in the category of graded $\Gamma$-comodules are given, on sequences of bounded-above comodules, by the graded local cohomology of the Steenrod algebra $\Gamma^*$ with coefficients in the limit of the adjoint $\Gamma^*$-modules.

For bounded-above comodules, these results reduce the calculation of the input of the Sadofsky and Hovey-Sadofsky spectral sequences to the calculation of graded local cohomology of the Steenrod algebra. Graded local cohomology is already a relatively familiar and well-studied theory, and there are already some computational tools for it, so we regard this identification of the Sadofsky and Hovey-Sadofsky spectral sequence $E_2$-terms as the main selling point of this paper.

We owe the reader some explanation of the peculiarities of {\em graded} local cohomology. Classically, given a commutative ring $R$, an ideal $I$ of $R$, and an $R$-module $M$, the local cohomology of $M$ at the ideal $I$ is defined by \begin{equation}\label{local coh old} H^*_I(M) := \underset{j\rightarrow\infty}{\colim}\Ext_R(R/I^j,M).\end{equation} In the literature (see \cite{MR1428799}, for example), when $R$ is a nonnegatively-graded Noetherian ring, not necessarily commutative, the {\em graded} local cohomology of a left $R$-module $M$ is defined as the colimit 
\begin{equation}\label{graded local coh}
 \underset{j\rightarrow\infty}{\colim}\Ext_R(R/I_j,M),
\end{equation}
where $I_j$ is the ideal of $R$ generated by all homogeneous elements of degree $\geq j$. The Noetherian hypothesis yields the isomorphism of \eqref{graded local coh} with $\underset{j\rightarrow\infty}{\colim}\Ext_R(R/I^j,M)$, where $I = I_1$ is the ideal of all elements of $R$ of positive degree. The Steenrod algebra, however, is graded but not Noetherian, and there is no guarantee that \eqref{graded local coh} must coincide with $\underset{j\rightarrow\infty}{\colim}\Ext_R(R/I^j,M)$. In this paper, we adopt \eqref{graded local coh}, rather than \eqref{local coh old}, as the definition of graded local cohomology for non-Noetherian graded rings like the Steenrod algebra.

Here are the other main results in this paper:
\begin{enumerate}
\item \Cref{Local cohomology functors associated...} covers elementary notions of local cohomology associated to a set of (left) ideals in a (not necessarily commutative) ring $R$. While local cohomology is well-studied over commutative rings, for the purposes of this paper we must consider local cohomology of {\em non}-commutative rings like the Steenrod algebra. Some familiar properties of local cohomology of commutative rings fail in surprising and treacherous ways in the absence of commutativity. Consequently we have to spend a large portion of this paper developing some necessary theory of local cohomology of noncommutative rings.

Given a set $S$ of ideals in $R$, we can consider the group \[ h^0_S(M)\coloneqq \underset{I\in S}{\colim} \hom_R(R/RI,M)\] of elements of an $R$-module $M$ which are $I$-torsion with respect to some ideal $I\in S$, or we can consider the $R$-submodule $H^0_S(M)$ of $M$ generated by the subgroup $h^0_S(M)$ of $M$. When $R$ is commutative, the functors $h^0_S$ and $H^0_S$ coincide, but for some noncommutative rings and some choices of $S$, they do not agree. Examples are given in Remark \ref{remark on ideal sets 2}. While $h^0_S$ is more familiar and computationally accessible, it is $H^0_S$ that carries the essential data about the relationship between comodules and modules, as we describe below.
\item \Cref{Comodules as a pretorsion...} reviews the well-known covariant embedding (via the adjoint action) of $\Gamma$-comodules into $\Gamma^*$-modules, and then reviews well-known ideas from torsion theory, and proves some preliminary results on the relationship between the two. A $\Gamma^*$-module $M$ is called {\em rational} if $M$ is in the image of this embedding. 
The main idea in \cref{Comodules as a pretorsion...} is that, given a coalgebra $\Gamma$ over a field, one can define a certain set $\dist(\Gamma)$ of ideals in $\Gamma^*$, the {\em distinguished ideals}, such that a $\Gamma^*$-module is rational if and only if the natural map $H^0_{\dist(\Gamma)}(M) \rightarrow M$ is an isomorphism. This is the content of Theorem \ref{theta-rationality and torsion}. 
One consequence is Corollary \ref{H0 is left exact}, which establishes that the functor $H^0_{\dist(\Gamma)}$ is left exact. (It is also true, but much easier to prove, that $h^0_{\dist(\Gamma)}$ is left exact.) Consequently we have some familiar homological tools for dealing with the right derived functors $H^*_{\dist(\Gamma)} \coloneqq R^*H^0_{\dist(\Gamma)}$, which we call {\em distinguished local cohomology.} Since the covariant embedding of comodules into modules is full and faithful, Theorem \ref{theta-rationality and torsion} characterizes $\Gamma$-comodules in terms of $\Gamma^*$-modules: the category of $\Gamma$-comodules is equivalent to the full subcategory of the $\Gamma^*$-modules generated by those $\Gamma^*$-modules $M$ such that $H^0_{\dist(\Gamma)}(M) \rightarrow M$ is an isomorphism.

I do not know of anywhere where Theorem \ref{theta-rationality and torsion} already appears in the literature, and I have not heard others mention the idea. Still I am uneasy about calling it a novel result: the ideas are quite close to those in sections 7 and 41 of \cite{MR2012570}. What is accomplished in \cref{Comodules as a pretorsion...} of this paper is only the development of a bit of (largely formal) theory, and the use of that bit of theory alongside ideas from \cite{MR2012570}. I regard \cref{Comodules as a pretorsion...} as a section mostly devoted to background and review. The hard ``non-formal'' work in this paper does not really begin until \cref{Distinguished local cohomology...}.
\item With \cref{Distinguished local cohomology...}, we begin to prove new results, at the cost of having to narrow the level of generality. Under the assumption that $\Gamma$ is a finite-type\footnote{To be clear, in this paper we adhere to the usual convention in algebraic topology: a graded vector space is said to be {\em finite-type} if it is finite-dimensional in each degree.} coalgebra over a field whose dual algebra $\Gamma^*$ is connected, Theorem \ref{main thm 1} shows that the rational graded $\Gamma^*$-modules (i.e., the graded $\Gamma$-comodules) form a hereditary pretorsion class in the category of graded $\Gamma^*$-modules. The same theorem furthermore shows that the higher distinguished local cohomology groups vanish on bounded-above graded $\Gamma^*$-modules, and that the bounded-above graded $\Gamma^*$-modules are rational. 

Most importantly, Theorem \ref{main thm 1} shows (under the same assumptions) that for any integer $n$ and any set $\{ M_i: i\in I\}$ of bounded-above graded $\Gamma$-comodules, the $n$th right-derived functor $R^n\prod_{i\in I}^{\Gamma} M_i$ of product in the category of graded $\Gamma$-comodules agrees (as a graded $\Gamma^*$-module via the adjoint action) with the distinguished local cohomology $H^n_{\dist(\Gamma)}\left(\prod_{i\in I} \iota(M_i)\right)$, where $\iota$ is the covariant embedding of comodules into modules. The point is that $\prod_{i\in I} \iota(M_i)$ is the {\em ordinary, familiar} Cartesian product of graded $\Gamma^*$-modules, not the more obscure categorical product in comodules. Consequently, if we reduce the distinguished local cohomology groups $H^n_{\dist(\Gamma)}$ to some familiar homological invariants, like a colimit of $\Ext$-groups, then we have a practical means of calculating $R^n\prod_{i\in I}^{\Gamma} M_i$, and consequently of calculating the homology groups of infinite products of spectra, via the Hovey-Sadofsky spectral sequence. Much of the rest of the paper is devoted to proving that $H^n_{\dist(\Gamma)}$ is indeed a straightforward colimit of $\Ext$-groups in sufficient generality to apply when $\Gamma$ is the dual Steenrod algebra at any prime.

One consequence of Theorem \ref{main thm 1} is Theorem \ref{structure thm}: if $\Gamma^*$ is finite-type and connected over a field, then for each integer $n$, every graded $\Gamma^*$-module $M$ is canonically an extension of an $n$-co-connected rational graded $\Gamma^*$-module by an $n$-connective graded $\Gamma^*$-module $\conn_n(M)$. Furthemore, the higher distinguished local cohomology groups of $M$ depend only on those of $\conn_n(M)$. Under the same hypotheses, we then get Corollary \ref{modules are limits of comodules}, which states that every graded $\Gamma^*$-module is a limit of a Mittag-Leffler sequence of rational $\Gamma^*$-modules. We also get Corollary \ref{uniqueness of module cat}, which states that the only full subcategory of $\gr\Mod(\Gamma^*)$ which contains the rational $\Gamma^*$-modules and which is closed under kernels and countable products is $\gr\Mod(\Gamma^*)$ itself.
\item \Cref{Mitchell coalgebras...} introduces the notion of a {\em Mitchell coalgebra}, a coalgebra over a field admitting a certain kind of decomposition which is intended to resemble the decomposition of the dual Steenrod algebra $A_*$ as the limit of the coalgebras $A(n)_*$. In \cite{MR793186}, S. Mitchell identified certain self-duality and compatibility properties of this decomposition of $A_*$, and we include analogous properties as part of the definition of a Mitchell coalgebra in Definition \ref{def of mitchell condition}, because we find in Theorem \ref{main thm on mitchell coalgebras} that precisely these properties can be used to show that the two local cohomology functors $h^*_{\dist(\Gamma)}$ and $H^*_{\dist(\Gamma)}$ coincide over the dual of a Mitchell coalgebra $\Gamma$. 
As a consequence, the distinguished local cohomology $H^*_{\dist(\Gamma)}$---which, by Theorem \ref{main thm 1}, computes the derived functors of product in the category of $\Gamma$-comodules---has the good computational properties of $h^*_{\dist(\Gamma)}$, and in particular, it is isomorphic to a colimit of $\Ext$-groups.

Since Mitchell's results in \cite{MR793186} establish that the dual Steenrod algebra is what we call a Mitchell coalgebra, we get the main results of this paper: if $\Gamma$ is the dual Steenrod algebra at any prime, then Corollary \ref{main cor 10} establishes that, for any graded $\Gamma^*$-module $M$, the distinguished local cohomology $H^*_{\dist(\Gamma)}(M)$ is isomorphic to the graded local cohomology $\underset{j\rightarrow\infty}{\colim} \Ext_{\Gamma^*}^*\left( \Gamma^*/I_j,M\right)$ of the Steenrod algebra $\Gamma^*$. A consequence is Corollary \ref{derived products are local cohomology}: if $\{ M_i: i\in I\}$ is a set of bounded-above graded $\Gamma$-comodules, then the $n$th derived functor $R^n\prod^{\Gamma}_{i\in I} M_i$ of product in the category of graded $\Gamma$-comodules is isomorphic, as an abelian group, to $\underset{j\rightarrow\infty}{\colim} \Ext_{\Gamma^*}^n\left( \Gamma^*/I_j,\prod_i M_i\right)$. Here $\prod_i M_i$ is the product (in the category of {\em modules}, i.e., the Cartesian product) of the graded $\Gamma^*$-modules $M_i$ with the adjoint action of $\Gamma^*$.
\item An abelian category satisfies Grothendieck's axiom $AB4^*$ if and only if products exist and are exact in that category. Weakenings of $AB4^*$ have been studied: given an integer $n$, an abelian category $\mathcal{C}$ is said to satisfy axiom $AB4^*\mhyphen (n)$ if and only if the $m$th derived functor of product in $\mathcal{C}$ vanishes for all $m > n$. The condition $AB4^*\mhyphen (0)$ is equivalent to $AB4^*$. When $\Gamma^*$ is the Steenrod algebra at some prime, one knows (e.g. from experience with the Adams spectral sequence) that $\Ext_{\Gamma^*}^{n}(M,N)$ is capable of being nonzero for arbitrarily large $n$, but this does not rule out the possibility of a finite bound on the integers $n$ such that $\underset{j\rightarrow\infty}{\colim} \Ext_{\Gamma^*}^{n}(\Gamma^*/I_j,N)$ is nonzero for some $N$. In other words, it seems plausible that the category of graded $\Gamma$-comodules satisfies axiom $AB4^*\mhyphen (n)$ for some $n$, and consequently that there is a uniform horizontal vanishing line in the $E_2$-term for the Hovey-Sadofsky spectral sequence calculating the mod $p$ homology of infinite products of spectra.

The purpose of \cref{Bounds on distinguished...} is to show that we are not, in fact, so lucky: Corollary \ref{no ab4n for dual steenrod alg} shows that the category of comodules over the dual Steenrod algebra, at any prime, cannot satisfy axiom $AB4^*\mhyphen (n)$ for any $n$ at all. This is a consequence of Theorem \ref{main thm 4}, which shows that, if $\Gamma$ is a coalgebra over a field satisfying some appropriate hypotheses (satisfied for the dual Steenrod algebras), then the category of bounded-above graded $\Gamma$-comodules satisfies $AB4^*\mhyphen (n)$ for some $n$ if and only if it satisfies $AB4^*$. The key to applying Theorem \ref{main thm 4} to the dual Steenrod algebra is a theorem of Margolis: Margolis has proven that the Steenrod algebras (more generally, the ``$\mathcal{P}$-algebras'' in Margolis' sense) have the properties assumed in the hypotheses of Thoerem \ref{main thm 4}. We also provide \cref{Review of Margolis...}, which is devoted to review of Margolis' basic theorems on graded modules over $\mathcal{P}$-algebras.
\item While Theorem \ref{main thm 1} shows that $R^n\prod_{i\in I}^{\Gamma} M_i$ agrees with the distinguished local cohomology group $H^n_{\dist}(\prod_{i\in I}\iota M_i)$, one would like to have a similar theorem for calculating more general derived limits, not just derived products, in comodule categories. \Cref{Derived functors of seq...} addresses that problem, at least for sequential limits. Theorem \ref{derived sequential lim thm} establishes that, when $\Gamma$ is a finite-type coalgebra over a field such that the dual algebra $\Gamma^*$ is connected, the derived functors of sequential limit $R^*\lim^{\Gamma}_iM_i$ in the category of graded $\Gamma$-comodules agree with the distinguished local cohomology groups $H^*_{\dist(\Gamma)}(\lim_i \iota M_i)$ as long as the sequence of graded $\Gamma^*$-modules $\dots \rightarrow \iota M_2 \rightarrow \iota M_1 \rightarrow \iota M_0$ has vanishing $\lim^1$, and as long as each comodule $M_i$ is bounded-above. It is an important and convenient point that this $\lim^1$-vanishing condition is checked in the {\em module} category, not in the {\em comodule} category, so it is relatively easy to check: one can simply verify that the Mittag-Leffler condition holds, for example.

Consequently, we have Corollary \ref{steenrod alg seq cor}: given a sequence $\dots \rightarrow M_2 \rightarrow M_1 \rightarrow M_0$ of bounded-above graded comodules over the dual Steenrod algebra at any prime, if the sequence of adjoint modules $\dots \rightarrow \iota M_2 \rightarrow \iota M_1 \rightarrow \iota M_0$ has vanishing $\lim^1$, then the derived limit $R^*\lim^{\Gamma}_i M_i$ in the category of comodules over the dual Steenrod algebra agrees with the graded local cohomology $\underset{j\rightarrow\infty}{\colim}\Ext^*_{\Gamma^*}\left(\Gamma^*/I_j,\lim_i \iota M_i\right)$ of the Steenrod algebra $\Gamma^*$.
\item \Cref{Review of Margolis...} reviews some results of Margolis, from \cite{MR738973}, on module theory over $\mathcal{P}$-algebras such as the Steenrod algebras. Those results of Margolis are used in several proofs in this paper, so it is useful to have an appendix dedicated to their review. There are no new results in \cref{Review of Margolis...}. 
\end{enumerate}

To demonstrate how the theory developed in this paper works in practice, we give two illustrative examples.
\begin{example}
Let $k$ be a field, and let $\Gamma$ be the graded $k$-coalgebra with $k$-linear basis $x_0, x_1, x_2, \dots$, with $x_i$ in degree $-2i$, and coproduct $\Delta(x_i) = \sum_{j=0}^i x_j\otimes x_{i-j}$. Of course the dual $k$-algebra $\Gamma^*$ is isomorphic to the polynomial algebra $k[x]$ with $x$ in degree $2$. As a consequence of Theorem \ref{main thm 1}, the $n$th right derived product $R^n\prod^{\Gamma}_i M_i$ of a set $\{ M_i: i\in I\}$ of graded $\Gamma$-comodules is the local cohomology group $H^n_{(x)}(\prod_{i\in I} \iota M_i)$, where $\iota M_i$ is $M_i$ regarded as a $k[x]$-module via the adjoint action. 

For example, in the case that $I = \mathbb{N}$ and $M_i$ is the $2i$th suspension of the graded $\Gamma$-subcomodule of $\Gamma$ which is $k$-linearly spanned by $x_0, \dots ,x_i$, we have $\iota(M_i) \cong k[x]/x^{i+1}$. While $H^n_{(x)}(M_i) \cong 0$ for each $n>0$ and each $i$, upon taking the product of the graded $\Gamma^*$-modules $k[x]/x^{i+1}$ for all $i\in \mathbb{N}$, we find {\em non-$x$-power-torsion} homogeneous elements like the degree $0$ homogeneous element $(1,1,1, \dots)$. Consequently the localization $x^{-1}\prod_{i\in \mathbb{N}} \iota M_i$ is nontrivial. Hence the first local cohomology group
\begin{align}
\label{iso 34093059} H^1_{(x)}\left( \prod_{i\in I} \iota M_i\right)
  &\cong \coker \left( \prod_{i\in \mathbb{N}} \iota M_i \rightarrow x^{-1}\prod_{i\in \mathbb{N}} \iota M_i\right) 
\end{align}
is nontrivial. As a consequence of Theorem \ref{main thm 1}, \eqref{iso 34093059} is isomorphic to the underlying graded $\Gamma^*$-module $\iota R^1\prod^{\Gamma}_{i\in I}M_i$ of the first right-derived product $R^1\prod^{\Gamma}_{i\in I}M_i$ of the product of the graded $\Gamma$-comodules $M_i$. 

Since $(x)$ is a principal ideal, the local cohomology groups $H^n_{(x)}(M)$ vanish for all $M$ and all $n>1$. Hence, as another consequence of Theorem \ref{main thm 1}, $R^n\prod^{\Gamma}_{i\in I}N_i$ is trivial for all $n>1$ and all sets $\{ N_i\}$ of graded $\Gamma$-comodules.
\end{example}
The above example is a bit facile, since local cohomology over a commutative ring of Krull dimension $1$, like $k[x]$, is very simple. The next example is more subtle, instead involving local cohomology of a ring of Krull dimension {\em zero}, so the ungraded local cohomology is trivial, but the {\em graded} local cohomology can be (and is) nonzero.
\begin{example}
Let $k$ be a field, and let $\Gamma^*$ be any finite-type graded $k$-algebra which is connected, i.e., trivial in negative degrees and isomorphic to $k$ in degree zero. As an illuminating example, we suggest 
\begin{align*}
 \Gamma^* &= k[y_1, y_2, y_3, \dots ]/(y_1^p,y_2^p,y_3^p, \dots),
\end{align*}
with $k$ of characteristic $p>0$, with each $y_i$ primitive, and with all $y_i$ in distinct degrees (for example, when $p=2$, this algebra is isomorphic to the subalgebra $E$ of the Steenrod algebra generated by the Milnor primitives); or for a slightly different example, simply let $\Gamma^*$ be the Steenrod algebra. Since $\Gamma$ is finite-type, we may write $\Gamma$ for the dual coalgebra $(\Gamma^*)^*$ without risk of confusion. {\em Recall our convention that we use cohomological gradings throughout, so that $\Gamma$ is concentrated in degrees $\leq 0$.}

Now let $\Gamma^{\geq -i}$ denote the graded $\Gamma$-subcomodule of $\Gamma$ consisting of all elements of degree $\geq -i$. We continue to write $\iota\Gamma$ to denote $\Gamma$ equipped with the adjoint $\Gamma^*$-action. Then the graded $\Gamma^*$-module $\iota\Gamma$ is isomorphic to the linear dual $\Gamma^{**}$ of $\Gamma^*$, hence is injective in the category of graded $\Gamma^*$-modules (e.g. see Proposition 11.3.12 of \cite{MR738973}). Products of injectives remain injective, so the middle term in the short exact sequence of graded $\Gamma^*$-modules
\begin{equation}\label{ses 031} 0 \rightarrow \prod_{i\geq 0} \Sigma^i\iota(\Gamma^{\geq -i})  \rightarrow \prod_{i\geq 0} \Sigma^{i}\iota\Gamma \rightarrow \prod_{i\geq 0} \Sigma^{i}\iota(\Gamma/\Gamma^{\geq -i})  \rightarrow 0\end{equation}
is injective. Applying the graded local cohomology functor $H^*$ to \eqref{ses 031} then yields an isomorphism for all negative integers $n$:
\begin{align}
\label{iso 0039558}
 H^{1,n}\left( \prod_{i\geq 0} \Sigma^i\iota\Gamma^{\geq -i}\right) &\cong \left\{ (\gamma_{0}, \gamma_{1},\gamma_{2}, \dots ) : \gamma_i\in \Gamma^{n-i}\right\}/Tors^n,
\end{align}
where the notation is as follows:
\begin{itemize}
\item $H^{1,n}\left( \prod_{i\geq 0} \Sigma^i\iota\Gamma^{\geq -i}\right)$ denotes the degree $n$ summand in the graded local cohomology group $H^{1}\left( \prod_{i\geq 0} \Sigma^i\iota\Gamma^{\geq -i}\right)$,
\item and $Tors^n$ denotes the subgroup of $\left(\prod_{i\geq 0}\Sigma^{i}\iota\Gamma\right)^n = \left\{ (\gamma_0, \gamma_{1}, \dots ) : \gamma_i\in \Gamma^{n-i}\right\}$ consisting of all those sequences $(\gamma_0, \gamma_{1}, \dots )$ such that there exists an integer $m$ such that every homogeneous element of $\Gamma^*$ of degree $\geq m$ acts trivially on $(\gamma_0, \gamma_{1}, \dots )$.
\end{itemize}
As a consequence of the main theorems of this paper, we have \begin{align*} H^{1}\left( \prod_{i\geq 0} \Sigma^i\iota\Gamma^{\geq -i}\right) &\cong \iota R^1\prod^{\Gamma}_{i\geq 0} \Sigma^i \Gamma^{\geq -i}.\end{align*} Hence, by the argument used to establish \eqref{iso 0039558}, the first right-derived product $R^1\prod^{\Gamma}$ is nontrivial whenever there exists a sequence $(\gamma_0,\gamma_1,\dots)$ of homogeneous elements of $\Gamma$, with monotone strictly decreasing degrees, and such that, for each integer $m$, some element of $\Gamma^*$ of degree $\geq m$ acts nontrivially on some $\gamma_i$.

For example, if $\Gamma$ is the dual Steenrod algebra, the sequence $(\xi_1,\xi_2,\xi_3,\dots)$ represents a nonzero element in $R^1\prod^{\Gamma}_{i\geq 1} \Sigma^{2(p^i-1)-1}\Gamma^{\geq -2(p^i-1)+1}$. The $\Gamma$-comodule $\Gamma^{\geq -2(p^i-1)+1}$ arises as the homology of an appropriate skeleton of the Eilenberg-Mac Lane spectrum $H\mathbb{F}_p$, so this example demonstrates an explicit nontrivial element on the $1$-line of the $E_2$-page of the Hovey-Sadofsky spectral sequence \eqref{sadofsky ss 2}.
\end{example}

\begin{remark}
Many cases of the results of this paper can be interpreted as having stack-theoretic content. Suppose we are given a commutative Hopf algebra $\Gamma$ over a commutative ring $A$. If $\Gamma$ is smooth over $A$, then the category of $\Gamma$-comodules is equivalent to the category of quasicoherent $\mathcal{O}_{B\mathbb{G}}$-modules, where $\mathcal{O}_{B\mathbb{G}}$ is the structure sheaf of the fppf site of the Artin stack $B\mathbb{G}$ of $\mathbb{G}$-torsors. Here $\mathbb{G}$ is the group scheme represented by $\Spec\Gamma$. In light of this, the question ``For what $n$ does the $n$th right derived functor $R^n\prod^{\Gamma}: \Comod(\Gamma)^I\rightarrow\Comod(\Gamma)$ vanish?'' becomes the question ``For what $n$ does the $n$th right derived functor $R^n\prod: QC\Mod(B\mathbb{G})^I\rightarrow QC\Mod(B\mathbb{G})$ vanish?'' 
This latter question was investigated for Deligne-Mumford stacks in \cite{hogadixu2009}, so the present paper could be seen as, in part, handling for certain Artin stacks the same problem that was investigated for Deligne-Mumford stacks in \cite{hogadixu2009}. But of course our motivations and main applications in this paper are really topological, rather than algebro-geometric. 

Artin stacks are more general than Deligne-Mumford stacks, so the results in this paper are not special cases of those in \cite{hogadixu2009}. Indeed, the results and the methods obtained in this paper are entirely different from the results and the methods involved in \cite{hogadixu2009}. 
\end{remark}

\begin{remark}
This paper examines the derived functors of products in categories of comodules over a coalgebr{\em a}, not a more general coalgebr{\em oid}, for example a Hopf algebroid. Given a suitable generalized homology theory $E_*$, see \cite{MR2337861} for a version of the Hovey-Sadofsky spectral sequence whose input is the derived functor of product in the category of graded $(E_*,E_*E)$-comodules, where now $(E_*,E_*E)$ is a Hopf algebroid and not typically a Hopf algebra (or coalgebra). So there is good topological motivation to try to prove Hopf algebroid analogues of the results in the present paper. A similar remark is also true with sequential limits in place of products throughout. 
\end{remark}

\begin{conventions}\label{conventions}\leavevmode
\begin{itemize}
\item Unless otherwise specified, our modules will be left modules, and our comodules will be right comodules.
\item All gradings in this paper are $\mathbb{Z}$-gradings whenever not otherwise stated.
\item When speaking of a {\em graded} module $M$, if we say that $M$ is injective, we mean that $M$ is injective in the category of {\em graded} modules. This is weaker than saying that $M$ is injective in the category of {\em ungraded} modules: see for example Remarks 3.3.11 in section I.3 of \cite{MR551625}. 
\item When we have a ring $R$ and a left ideal $I$ of $R$, we will write $RI$ for the ideal $I$ regarded as a left $R$-module. Consequently $R/RI$ denotes the left $R$-module given by the cokernel of the inclusion $RI \hookrightarrow R$. While writing $R/RI$ rather than $R/I$ may seem annoying to some readers, it is a common convention, e.g. as in \cite{lam2013first}, and this convention avoids some notational ambiguities: for example, when $A$ is the Steenrod algebra, if we were to write $\hom_A(A/(\Sq^1,\Sq^2,\dots),M)$, it could leave the reader uncertain whether $A/(\Sq^1,\Sq^2,\dots)$ means $A$ modulo the two-sided ideal generated by $\{\Sq^1,\Sq^2,\dots\}$, or the much larger quotient module given by $A$ modulo only the {\em left} ideal generated $\{\Sq^1,\Sq^2,\dots\}$. It is convenient to have the notation $A/A(\Sq^1,\Sq^2,\dots)$ reserved for the latter meaning.
\item At many places in this paper, we give arguments involving annihilators over noncommutative rings, such as Steenrod algebras. These arguments require a bit of care, because some of the nice behavior of annihilators in commutative algebra does not carry over to the noncommutative setting. 
We will follow the notational conventions for annihilators from \cite{lam2013first}: if $R$ is a ring and $M$ is a (left) $R$-module, then $\ann(M)$ denotes the set $\{ r\in R: rm=0\ \forall m\in M\}$, which is a {\em two-sided} ideal in $R$. If $S$ is a subset of $M$, we write $\ann_{\ell}(S)$ for the subset $\{ r\in R: rs=0\ \forall s\in S\}$, which is in general only a {\em left} ideal of $R$. Of course this distinction is only necessary when $R$ is noncommutative. 

In particular, if $M$ is a cyclic left $R$-module, then $M$ is isomorphic to $R/R\ann_{\ell}(g)$, where $g$ is a generator for $M$. We have an equality $R/R\ann(M) = R/R\ann_{\ell}(M)$, but $R/R\ann(M)$ and $R/R\ann_{\ell}(M)$ are {\em not} necessarily isomorphic to $R/R\ann_{\ell}(g)$, hence not necessarily isomorphic to $M$. The discussion in section 2.4 of \cite{lam2013first} is an excellent textbook treatment of this issue. 
\item Given a graded ring $R$ and graded $R$-modules $M$ and $N$, we write $\hom_R(M,N)$ for the degree-preserving $R$-linear morphisms $M\rightarrow N$, and we write $\underline{\hom}_R(M,N)$ for the graded abelian group whose degree $n$ summand\footnote{Some references, e.g. \cite{brunerprimer}, use the opposite grading on $\underline{\hom}_R$---hence the need to give our grading conventions explicitly. The argument for our choice of gradings is that it is the unique one so that $\hom_R(L,\underline{\hom}_R(M,N)) \cong \hom_R(L\otimes_R M,N)$.} is $\hom_R(\Sigma^n M,N)$.
\item 
Throughout, when we have a coalgebra over a commutative ring, we use the standard notations from the theory of Hopf algebroids (from the influential first appendix of \cite{MR860042}, for example): we write $A$ for the commutative ring, and we write $\Gamma$ for the coalgebra. We write $\Gamma^*$ for the $A$-linear dual algebra of $\Gamma$. This includes our discussions of the Steenrod algebra as the motivating example for the theory: we will write $\Gamma$ for the {\em dual} Steenrod algebra, and consequently $\Gamma^*$ for the Steenrod algebra itself.

When we have a finite-type graded coalgebra $\Gamma$, we {\em always} set up the grading so that its dual algebra $\Gamma^*$ is connected\footnote{``Connected'' is a standard term, but to avoid any possible misunderstanding, we include its definition here: to say that the finite-type algebra $\Gamma^*$ is connected is to say that, for each $n$, the degree $n$ summand of $\Gamma^*$ (equivalently, $\Gamma$) is a finite dimensional $A$-vector space, that $\Gamma^*$ is one-dimensional in degree $0$, and that $\Gamma^*$ is trivial in negative degrees.}. Whenever possible, we avoid direct manipulation of the grading on $\Gamma$ itself, in order to completely avoid any ambiguity or confusion about the effect of dualization on the gradings. 

One consequence is that $\iota(\Gamma)$ is concentrated in non{\em positive} degrees, where $\iota$ is the {\em covariant} embedding of graded $\Gamma$-comodules into graded $\Gamma^*$-modules. This fact becomes very useful, starting in Theorem \ref{main thm 1}.
\item We write $\prod$ for a product of graded modules, and $\prod^{\Gamma}$ for the product in the category of graded $\Gamma$-comodules.
\end{itemize}
\end{conventions}

We are grateful to Gabe Angelini-Knoll and Bob Bruner for many conversations about derived limits in comodules and the Sadofsky spectral sequence, and an anonymous referee for helpful comments.

\section{Local cohomology functors associated to sets of ideals.}
\label{Local cohomology functors associated...}
This section consists of elementary notions about generalized local cohomology functors associated to sets of ideals in some (not necessarily commutative) ring. We do not claim originality for the ideas in this section, but we do not know anywhere where this sequence of ideas already appears in the literature.

We first define the preorder of ideal sets of a graded ring.
\begin{definition}\label{def of ideal sets}
Let $R$ be a graded ring. 
\begin{itemize}
\item By an {\em ideal set in $R$} we mean a set $S$ of homogeneous proper left ideals of $R$. 
\item If $S,S^{\prime}$ are ideal sets in $R$, we write $S\leq S^{\prime}$ if, for every element $I$ of $S$, there exists some element $J\in S^{\prime}$ such that $I\supseteq J$. (Intuitively, this says that $S\leq S^{\prime}$ iff $S^{\prime}$ is ``finer'' than $S$.) The relation $\leq$ on ideal sets is transitive and reflexive, hence the collection of ideal sets in $R$ is a preorder. We write $\idealsets(R)$ for this preorder.
\item We say that ideal sets $S$ and $S^{\prime}$ are {\em equivalent} if $S\leq S^{\prime}$ and $S^{\prime}\leq S$.
\end{itemize}
\end{definition}

\begin{definition}[Properties of ideal sets]
We continue to let $R$ be a graded ring. 
\begin{itemize}
\item We will say that an ideal set $S$ in $R$ is {\em connected} if $S$ is connected as a partially ordered set under inclusion, i.e., if $I,J\in S$, then there exists a finite sequence $I = I_0, J_0, I_1, J_1, \dots ,I_{n-1},J_{n-1},I_n,J_n = J$ of elements of $S$ such that $I_h \subseteq J_h$ and $I_{h+1}\subseteq J_h$ for all $h$.
\item We will say that an ideal set $S$ in $R$ is {\em filtered} 
if, for each $I,J\in S$, there exists an element of $S$ contained in both $I$ and $J$.
\end{itemize}
\end{definition}

A non-filtered ideal set at least has a canonically-associated filtered ideal set, its {\em filtered closure}:
\begin{definition}
Given an ideal set $S$ in $R$, let $\overline{S}$ denote the set of all intersections\footnote{To be clear: an element of $\overline{S}$ is a homogeneous left ideal in $R$ which is the intersection of finitely many homogeneous left ideals which are members of $S$.} of finite sets of members of $S$. We call $\overline{S}$ the {\em filtered closure} of $S$. 
\end{definition}
The filtered closure of $S$ is a filtered ideal set in $R$, and $S\leq \overline{S}$. Furthermore, $\overline{S}$ is minimal (up to equivalence) with that property. That is, if $T$ is any filtered ideal set in $R$ such that $S\leq T$, then $\overline{S}\leq T$.

\begin{definition}[Local cohomology theories associated to an ideal set]
 Given an ideal set $S$, we have {\em two} associated degree zero local cohomology functors.
\begin{enumerate}
\item First, we have the functor
\begin{align*}
 h^0_S: \gr\Mod(R) & \rightarrow \gr\Ab \\
        h^0_S(M) &= \underset{I\in S}{\colim}\ \underline{\hom}_R(R/RI,M). \end{align*}
In particular, if $S$ is filtered, then $h^0_S(M)$ is the graded subgroup of $M$ generated by all homogeneous elements which are $I$-torsion for some element $I$ of $S$. 

We write $h^n_S$ for the $n$th right derived functor $R^nh^0_S$ of $h^0_S$.
\item
Meanwhile, if $S$ is filtered, we also have the functor
\begin{align*}
 H^0_S: \gr\Mod(R) & \rightarrow \gr\Mod(R) \end{align*}
given by letting $H^0_S(M)$ be the graded left $R$-submodule of $M$ generated by the subgroup $h^0_S(M)$ of $M$.
We write $H^n_S$ for the $n$th right derived functor $R^nH^0_S$ of $H^0_S$.
\end{enumerate}
\end{definition}
Note that, if $S\leq S^{\prime}$ and $M$ is a graded $R$-module, then every homogeneous element of $M$ which is $I$-torsion for some $I\in S$ is also $J$-torsion for some $J\in S^{\prime}$. If $S^{\prime}$ is also assumed to be filtered, then we have a well-defined map from $\underline{\hom}_R(R/RI,M)$ to $\underset{J\in S^{\prime}}{\colim}\ \underline{\hom}_R(R/RJ,M)$. Consequently, if $S\leq S^{\prime}$ and $S^{\prime}$ is filtered, we get a canonical natural transformation $h^0_S \rightarrow h^0_{S^{\prime}}$.

We now offer a sequence of examples, non-examples, and observations to illustrate the various notions defined in Definition \ref{def of ideal sets}.
\begin{remark}[Why connectedness matters] \label{remark on ideal sets}
If an ideal set $S$ is not connected, we can still make the definition $h^0_S(M) \coloneqq \underset{I\in S}{\colim}\ \underline{\hom}_R(R/RI,M)$, but the resulting abelian group $h^0_S(M)$ can fail to be a subgroup of $M$. For example, let $k$ be a field, and let $R = k[x,y]$. Let $S$ be the ideal set $\left\{ (x),(y)\right\}$. Then $h^0_S(M)$ is the direct sum of the $x$-torsion subgroup of $M$ and the $y$-torsion subgroup of $M$. This is a subgroup of $M\oplus M$, but it is not a subgroup of $M$ itself in any natural way.

The filtered closure $\overline{S}$ of $S$ is $\{ (x), (y), (xy)\}$, so $h^0_{\overline{S}}(M)$ is the $(xy)$-torsion subgroup of $M$, i.e., the set of elements $m\in M$ such that $xym=0$. This {\em is} a subgroup of $M$, not merely a subgroup of $M\oplus M$. What we saw in this example is one case of a general phenomenon: a filtered partially-ordered set is automatically connected, so $h^0_{\overline{S}}(M)$ is a subgroup of $M$ for any ideal set $S$. 

This is also an example of an ideal set $S$ such that $h^0_S$ is not isomorphic to $h^0_{\overline{S}}$. In general, there is no reason to expect $h^0_S$ to coincide with $h^0_{\overline{S}}$\end{remark}
\begin{remark}[When $h^0_S$ and $H^0_S$ differ] \label{remark on ideal sets 2}
Suppose $S$ is filtered. 
If $R$ is commutative, then $\underline{\hom}_R(R/RI,M)$ is in fact an $R$-submodule of $M$, and not only a subgroup. Put another way, when $R$ is commutative, the $I$-torsion in a given $R$-module is a submodule and not merely a subgroup. Consequently the natural transformation $h^0_S\rightarrow H^0_S$ is an isomorphism for all filtered ideal sets $S$, when $R$ is commutative.

On the other hand, when $R$ is noncommutative, $h^0_S$ may differ from $H^0_S$. Here is an explicit example where $h^0_{S}(M)$ fails to be an $R$-submodule of $M$, and consequently $h^0_{S}(M)$ fails to agree with $H^0_{S}(M)$. Let $R$ be the subalgebra $A(1)$ of the mod $2$ Steenrod algebra generated by $\Sq^1$ and $\Sq^2$. Let $S$ be the left ideal in $A(1)$ generated by $\Sq^1$. Then $h^0_{S}(M)$ is simply the graded subgroup of $M$ consisting of the elements of $M$ annihilated by $\Sq^1$. 
In the case $M = A(1)$, we have $\Sq^1\in h^0_{S}(A(1))$, since $\Sq^1\Sq^1 = 0$ in $A(1)$. However, $\Sq^2\Sq^1 \notin h^0_{S}(A(1))$, since $\Sq^1\Sq^2\Sq^1 \neq 0$ in $A(1)$. So $h^0_{S}(A(1))$ is not closed under left $A(1)$-multiplication in $A(1)$, i.e., $h^0_{S}(A(1))$ is not a left $A(1)$-submodule of $A(1)$. Hence $h^0_{S}(A(1))$ fails to coincide with $H^0_{S}(A(1))$: the latter contains $\Sq^2\Sq^1$, while the former does not.
\end{remark}
\begin{remark}[Why equivalence matters] 
If filtered ideal sets $S,S^{\prime}$ in $R$ satisfy $S\leq S^{\prime}$, then the natural monomorphism $h^0_S(M)\hookrightarrow M$ factors uniquely through the natural monomorphism $h^0_{S^{\prime}}(M)\hookrightarrow M$. Consequently we get natural transformations $h^n_S\rightarrow h^n_{S^{\prime}}$ and $H^n_S\rightarrow H^n_{S^{\prime}}$ for each $n$.

If the filtered ideal sets $S$ and $S^{\prime}$ are equivalent, then we have $h^0_S(M) = h^0_{S^{\prime}}(M)$ as subgroups of $M$, so we get natural isomorphisms $h^n_S \cong h^n_{S^{\prime}}$ and $H^n_S \cong H^n_{S^{\prime}}$ for all $n$ as well.
\end{remark}
\begin{remark}[Left exactness of $h^0_S$ and $H^0_S$]
We have much better tools for understanding and calculating the right derived functors $h^*_S$ of $h^0_S$ when we know that $h^0_S$ is left exact. If the ideal set $S$ is filtered, then colimits over $S$ are exact, so
$h^0_S(-) = \underset{I\in S}{\colim}\ \underline{\hom}_R(R/RI,-)$ is indeed left exact. 

It is much less obvious that $H^0_S$ is left exact. This is a topic we take up later, using tools from torsion theory, in Corollary \ref{H0 is left exact}.
\end{remark}

We now give important examples of ideal sets.
\begin{example}[The ideal set of powers of an ideal] \label{examples of ideal sets} 
The famous case of $H^0_{S}(M)$ is the case where $R$ is commutative and concentrated in degree zero, $I$ is an ideal of $R$, and $S$ is the set $S = \{ I, I^2, I^3, \dots\}$ of powers of $I$. In that case, $H^0_{S}(M)$ is simply the classical local cohomology $H^0_I(M)$ of $M$ at the ideal $I$, in the sense of \cite{MR0222093} and \cite{MR3014449}, among many other references. 

The reason for introducing $h^0_{S}$ and $H^0_{S}$ in Definition \ref{def of ideal sets} is that, when generalizing local cohomology from the classical setting to a setting where $R$ is not necessarily commutative, $h^0_{S}$ and $H^0_{S}$ can differ. This may lead to confusion: if, for example, $R$ is noncommutative and $S$ is the set $\{ I,I^2,I^3, \dots\}$ of powers of some left ideal $I$ of $R$, then $h^0_{S}(M) = \underset{n\rightarrow\infty}{\colim}\ \underline{\hom}_R(R/RI^n,M)$ is generally only a subgroup of $M$, not necessarily an $R$-submodule of $M$, so some familiar arguments from classical local cohomology no longer work. For example, it no longer even makes sense to ask whether $h^0_{S}:\gr\Mod(R)\rightarrow\gr\Ab$ is idempotent. On the other hand, it {\em does} make sense to ask whether $H^0_{S}:\gr\Mod(R)\rightarrow\gr\Mod(R)$ is idempotent. The functor $H^0_S$ has some very desirable properties (see Theorem \ref{theta-rationality and torsion}, for example), but the derived functors $H^n_S$ (unlike $h^n_S$) are not generally isomorphic to colimits of $\Ext$-groups, so it can be much less clear how to actually calculate them. 
\end{example}
\begin{example}[The trivial ideal set]\label{examples of ideal sets 2} 
The following example is trivial, but helps to build intuition for the preorder $\idealsets(R)$. 
The preorder $\idealsets(R)$ has a maximal element, which can be taken to be the ideal set $\{ (0)\}$ consisting of simply the zero ideal of $R$. This ideal set is equivalent to any other ideal set in $R$ containing the zero ideal, for example the ideal set consisting of all proper homogeneous left ideals of $R$. For all graded $R$-modules $M$, we have $h^0_{\{(0)\}}(M) = M = H^0_{\{(0)\}}(M)$, so $h^n_{\{(0)\}}(M) = 0 = H^n_{\{(0)\}}(M)$ for all $n>0$. 

The point is that the maximal element in $\idealsets(R)$ is the element whose associated local cohomology theory has $h^0(M)$ consisting of {\em all of $M$}. The partial order $\leq$ on $\idealsets(R)$ has the property that {\em smaller filtered elements $S$ of $\idealsets(R)$ make $h^0_S(M)$ a smaller subgroup of $M$.}
\end{example}
\begin{example}[The ideal set of distinguished ideals of a module]
\label{examples of ideal sets 3} 
Let $R$ be a graded ring, and let $\Theta$ be a graded left $R$-module. Let $\homog(\Theta)$ be the set of homogeneous elements of $\Theta$, and for each $m\in \homog(\Theta)$, let $\ann_{\ell}(m)$ be the left annihilator of $m$. Let $\dist(\Theta)$ be the ideal set $\{ \ann_{\ell}(m) : 0\neq m \in \homog(\Theta)\}$ of $R$. We refer to the members of $\dist(\Theta)$ as the {\em strongly $\Theta$-distinguished ideals of $R$}, and we refer to the members of the filtered closure $\overline{\dist(\Theta)}$ as the {\em $\Theta$-distinguished ideals of $R$.}
\end{example}
\begin{example}[The ideal set $\grad$ of a graded ring]\label{examples of ideal sets 4}
Let $R$ be an $\mathbb{N}$-graded ring, and let $\grad$ be the ideal set $\{ I_1, I_2, I_3, \dots\}$ in $R$, where $I_n$ is the left ideal (equivalently, two-sided ideal) in $R$ generated by all homogeneous elements of degree $\geq n$. Note that $\grad$ is filtered. The ideal set $\grad$ will play a pivotal role in the proof of the most important result in this paper, Theorem \ref{main thm on mitchell coalgebras}.
\end{example}

\section{Comodules as a pretorsion class in modules.}
\label{Comodules as a pretorsion...}

This section covers preliminary notions on the covariant embedding of $\Gamma$-comodules into $\Gamma^*$-modules, as well as preliminary notions on (pre)torsion theories. 

\subsection{Review of the embedding of the category of $\Gamma$-comodules into the category of $\Gamma^*$-modules.}

Let $A$ be a commutative ring, and let $\Gamma$ be a graded $A$-coalgebra which is flat as an $A$-module. Let $\Gamma^*$ denote the $A$-linear dual graded algebra $\underline{\hom}_A(\Gamma,A)$ of $\Gamma$. 
 We have a well-known functor
\begin{equation}
 \label{comodule inclusion functor}
 \iota:\gr\Comod(\Gamma) \rightarrow \gr\Mod(\Gamma^*)
\end{equation}
given by sending a graded $\Gamma$-comodule $M$ to the graded $A$-module $M$, equipped with the ``adjoint'' $\Gamma^*$-action. The adjoint left $\Gamma^*$-action\footnote{To avoid possible confusion, we remark that when $\Gamma$ is not only a coalgebra but a finite-type Hopf algebra, then $\Gamma\cong \Gamma^{**}$ is also a graded $\Gamma^*$-module via the {\em contragredient} action, which differs from the adjoint action by an application of the antipode of $\Gamma^*$.} is the action $\Gamma^*\otimes_A M \rightarrow M$ given by sending $f\otimes m$ to the image of $m$ under the composite map 
\[ M \stackrel{\psi_M}{\longrightarrow} M\otimes_A \Gamma \stackrel{M\otimes f}{\longrightarrow} M\otimes_A A \stackrel{\cong}{\longrightarrow} M.\]
The graded $\Gamma^*$-modules in the essential image of the functor $\iota$ are called {\em rational} modules in the literature, e.g. in the textbook \cite{MR2012570}. The use of the term ``rational'' here is well-established, so we use that term in this paper\footnote{There is some risk of confusion here: topologists, like the author of this paper, are used to the term ``rational'' being reserved for abelian groups which are vector spaces over $\mathbb{Q}$, or spectra which are modules over $H\mathbb{Q}$, or other small variations on this theme. That use of the term ``rational'' is simply unlike the use of the term ``rational'' in the coalgebra literature and in this paper. 
Rational (in the sense of coalgebra) modules over the Steenrod algebra are one of the main objects of study in this paper, but they do not arise as the homology or cohomology of rational (in the sense of topology) spaces or spectra, except in trivial cases.}. 

The following theorem summarizes some established properties of the functor \eqref{comodule inclusion functor}; see 4.1, 4.3, 4.7, 7.1, and 20.1 of \cite{MR2012570} for a comprehensive treatment.
\begin{theorem}\label{review thm} The following claims are each true:
\begin{enumerate}
\item The functor $\iota$ is faithful and its image lies in the full subcategory of $\gr\Mod(\Gamma^*)$ consisting of all graded $\Gamma^*$-modules $M$ such that $M$ is a graded submodule of a graded quotient module of a coproduct of copies of suspensions of $\iota\Gamma$. This full subcategory of $\gr\Mod(\Gamma^*)$ is denoted $\sigma[{}_{\Gamma^*}\Gamma]$.
\item The resulting functor $\gr\Comod(\Gamma)\rightarrow \sigma[{}_{\Gamma^*}\Gamma]$ is full if and only if $\Gamma$ is locally projective as an $A$-module. Consequently, when $\Gamma$ is projective as an $A$-module, we may regard $\gr\Comod(\Gamma)$ as (via the functor $\iota$) a full subcategory of $\gr\Mod(\Gamma^*)$. 
\item If $\Gamma$ is projective as an $A$-module, then the functor $\iota$ is full, faithful, and admits a right adjoint $\tr: \gr\Mod(\Gamma^*)\rightarrow \gr\Comod(\Gamma)$, called the ``rational functor'' or the ``trace functor,'' and which is given on a graded $\Gamma^*$-module $M$ by letting $\tr(M)$ be the graded $\Gamma^*$-submodule $M$ generated by the homogeneous elements $m$ such that $m$ is in the image of a graded $\Gamma^*$-module homomorphism from a rational graded $\Gamma^*$-module to $M$.
\item The functor $\iota: \gr\Comod(\Gamma) \rightarrow \gr\Mod(\Gamma^*)$ is an equivalence of categories if and only if $\Gamma$ is finitely generated and projective as an $A$-module. 
\item In particular, suppose that $\Gamma$ is projective as an $A$-module. Then the rational graded $\Gamma^*$-modules form a full abelian subcategory of the graded $\Gamma^*$-modules, closed under kernels, cokernels, and coproducts. Furthermore, the following are equivalent:
\begin{itemize}
\item Every graded $\Gamma^*$-module is rational.
\item $\Gamma$ is finitely generated as an $A$-module. 
\end{itemize}
\end{enumerate}
\end{theorem}

In Proposition \ref{rationals closed under quots and subs}, we offer a proof of part of Theorem \ref{review thm}, in order to make it clear (as explained in Remark \ref{rationals closed under quots and subs 2}) that the proof generalizes from rational modules to $\Theta$-rational modules.
\begin{prop} \label{rationals closed under quots and subs}
Suppose that $\Gamma^*$ is projective as an $A$-module. Then the following claims are each true:
\begin{enumerate}
\item 
Every coproduct of rational graded $\Gamma^*$-modules is rational.
\item 
Every graded submodule of a rational graded $\Gamma^*$-module is rational.
\item
Every graded quotient of a rational graded $\Gamma^*$-module is rational.
\end{enumerate}
\end{prop}
\begin{proof}\leavevmode
\begin{itemize}
\item[(1) and (2)]
It follows immediately from the definition of $\sigma[{}_{\Gamma^*}\Gamma]$ that it is closed under coproducts and graded submodules. By the first part of Theorem \ref{review thm}, $\sigma[{}_{\Gamma^*}\Gamma]$ agrees with the category of rational graded $\Gamma^*$-modules.
\item[(3)]
Suppose that $Q$ is a graded quotient of a rational graded $\Gamma^*$-module $M$. Since $M$ is rational, it embeds into a quotient module $F^{\prime}$ of a coproduct $F$ of suspensions of copies of $\iota\Gamma$. We have the diagram of graded $\Gamma^*$-modules
\begin{equation}\label{diag 195949}\xymatrix{
 & M\ar@{->>}[ld] \ar@{^{(}->}[rd] & & F \ar@{->>}[ld] \\
 Q & & F^{\prime} &
}\end{equation}
and, taking the pushout of the subdiagram of \eqref{diag 195949} containing $M$ and $Q$ and $F^{\prime}$, we get a graded $\Gamma^*$-module into which $Q$ embeds, and which is also a quotient of $F$, since epimorphisms and monomorphisms are each stable under pushouts in a category of modules over a ring. Hence $Q$ is in $\sigma[{}_{\Gamma^*}\Gamma]$, hence $Q$ is rational.
\end{itemize}
\end{proof}

\begin{definition}
Let $R$ be a graded ring, and let $\Theta$ be a graded left $R$-module. We will call a graded left $R$-module $M$ {\em $\Theta$-rational} if $M$ is a graded submodule of a graded quotient module of a coproduct of copies of suspensions of $\Theta$.
\end{definition}
\begin{remark}\label{rationals closed under quots and subs 2}
Note that Proposition \ref{rationals closed under quots and subs} remains true, with the same proof, if we replace ``rational'' with ``$\Theta$-rational'' throughout.
\end{remark}

\subsection{Distinguished ideals and distinguished torsion.}
\label{Distinguished ideals and...}

Suppose that $\Gamma$ is projective over $A$. 
Then, as a consequence of Theorem \ref{review thm}, the graded $\Gamma$-comodules form a particularly nice subcategory of the graded $\Gamma^*$-modules. Here ``nice'' means, in particular, full and coreflective and abelian. It is not obvious, at a glance, how to tell if a given graded $\Gamma^*$-module actually lives in that nice subcategory. 
Consequently, one would like to have a purely module-theoretic characterization of that subcategory: that is, given a graded $\Gamma^*$-module $M$, we would like to be able to determine whether $M$ is rational by means of some criterion which refers only to the $\Gamma^*$-action on $M$. 
One of our tasks (completed in Theorem \ref{main thm 1}) in this section is to show that there is indeed such a criterion: it is the property of being {\em distinguished-torsion}, which we now define.

\begin{prop}\label{prop-def of distinguished}
Let $R$ be a graded ring, and let $\Theta$ be a graded left $R$-module. 
Let $I$ be a homogeneous left ideal of $R$. Then the following conditions are equivalent:
\begin{itemize}
\item There exists an exact sequence of graded left $R$-modules
\[ 0 \rightarrow I \stackrel{f}{\longrightarrow} R \rightarrow \Sigma^n \Theta \]
for some integer $n$, where $f$ is the canonical inclusion of $I$ into $R$.
\item $\Theta$ contains a graded $R$-submodule which is isomorphic to a suspension of $R/RI$.
\item $I$ is a member of the ideal set $\dist \Theta$ defined in Example \ref{examples of ideal sets 3}.
\end{itemize}
\end{prop}
\begin{proof}
Elementary.
\end{proof}
\begin{definition}\label{def of distinguished}
Let $R,\Theta$ be as in Proposition \ref{prop-def of distinguished}.
\begin{itemize}
\item
As in Example \ref{examples of ideal sets 3}, we call a left ideal $I$ of $R$ {\em strongly $\Theta$-distinguished} if it is homogeneous and satisfies the equivalent conditions of Proposition \ref{prop-def of distinguished}. Similarly, we call a homogeneous left ideal $I$ {\em $\Theta$-distinguished} if it contains the intersection of a finite set of strongly $\Theta$-distinguished left ideals. 
\item Let $\dist \Theta$ be the ideal set of $\Theta$-distinguished left ideals of $R$. As in Definition \ref{def of ideal sets}, we write $\overline{\dist{\Theta}}$ for the filtered closure of $\dist{\Theta}$. 
We say that a left $R$-module $M$ is {\em $\Theta$-distinguished-torsion} if the inclusion $H^0_{\overline{\dist \Theta}}(M) \hookrightarrow M$ is an isomorphism. 
\item
The most important case of these notions is when $R = \Gamma^*$, the dual $A$-algebra of an $A$-coalgebra $\Gamma$ which is projective as an $A$-module, and $\Theta = \iota\Gamma$. In that case we write {\em distinguished ideal}, {\em distinguished torsion}, $h^0_{\dist}$, and $H^0_{\dist}$ as shorthand for $\iota\Gamma$-distinguished ideal, $\iota\Gamma$-distinguished torsion, $h^0_{\overline{\dist(\iota\Gamma)}}$, and $H^0_{\overline{\dist(\iota\Gamma)}}$, respectively.
\end{itemize}
\end{definition}

\begin{theorem}\label{theta-rationality and torsion}
Let $R$ be a graded ring and let $\Theta$ be a graded left $R$-module. Let $M$ be a graded left $R$-module. Then the following are equivalent:
\begin{enumerate}
\item $M$ is $\Theta$-rational.
\item For every homogeneous $m\in M$, the annihilator of $m$ is a $\Theta$-distinguished left ideal of $R$.
\item $M$ is $\Theta$-distinguished torsion.
\end{enumerate}
\end{theorem}
\begin{proof}
It is easy to see from the definition of distinguished torsion that conditions 2 and 3 are equivalent, so the rest of this proof consists of showing that conditions 1 and 2 are equivalent.
\begin{itemize}
\item 
Suppose that $M$ is $\Theta$-rational, and that $m$ is a homogenous element of $M$.
Embed $M$ into a quotient $Q$ of a coproduct of copies of $\Theta$ via a map $f: M \rightarrow Q$.
Since $f$ is injective, we have $\ann_{\ell}(f(m)) = \ann_{\ell}(m)$.
Now choose a surjective left $R$-module map $\sigma: F \rightarrow Q$ where $F$ is a coproduct of
copies of $\Theta$, and choose a homogeneous element $m^{\prime}$ in $F$ such that 
$\sigma(m^{\prime}) = f(m)$.
Since coproducts in categories of modules are direct sums, the element $m^{\prime}\in F$ has only finitely many nonzero components in summands $\Theta$ of $F$. Write $m^{\prime}_1, \dots ,m^{\prime}_n$ for these homogeneous nonzero elements of $\Theta$. 
We now have the chain of equalities and containments
\begin{align*}
 \bigcap_{i=1}^n \ann_{\ell}(m^{\prime}_i) &= \ann_{\ell}(m^{\prime}) \\
  &\subseteq \ann_{\ell}(\sigma(m^{\prime})) \\
  &= \ann_{\ell}(f(m)) \\
  &= \ann_{\ell}(m),
\end{align*}
so $\ann_{\ell}(m)$ contains the intersection $\bigcap_{i=1}^n \ann_{\ell}(m^{\prime}_i)$ of the left annihilators of a finite set of homogeneous elements of $\Theta$, i.e., $\ann_{\ell}(m)$ is a $\Theta$-distinguished left ideal of $R$. 
\item
Suppose conversely that the annihilator of every homogeneous element of $M$ is a $\Theta$-distinguished left ideal of $R$. Let $\homog(M)$ denote the set of homogeneous elements of $M$, and consider the graded $R$-module morphism \[ \epsilon_M: \coprod_{m\in \homog(M)} \Sigma^{\left| m\right|} R/R\ann_{\ell}(m) \rightarrow M\]
which sends the summand $\Sigma^{\left| m\right|} R/R\ann_{\ell}(m)$ corresponding to $m\in \homog(M)$ to $M$ via the map $\Sigma^{\left| m\right|} R/R\ann_{\ell}(m) \rightarrow M$ sending $1$ to $m$.
For each $m\in \homog(M)$, choose a finite set $\gamma_{m,1}, \dots ,\gamma_{m,n_m}$ of homogeneous elements of $\Theta$ such that $\ann_{\ell}(m)$ contains $\cap_{i=1}^{n_m}\ann_{\ell}(\gamma_{m,i})$.
Then $R/R\left( \cap_{i=1}^{n_m}\ann_{\ell}(\gamma_{m,i})\right)$ surjects on to $R/R\ann_{\ell}(m)$.
We also have the graded $R$-module monomorphism 
\begin{equation}\label{map 30499949ff} R/R\left(\cap_{i=1}^{n_m}\ann_{\ell}(\gamma_{m,i}) \right)
  \rightarrow \coprod_{i=1}^{n_m} R/R\ann_{\ell}(\gamma_{m,i}) \end{equation}
which sends any given element $x\in R/R\left(\cap_{i=1}^{n_m}\ann_{\ell}(\gamma_{m,i})\right)$ to the 
element 
\[ \left( x\mod \ann_{\ell}(\gamma_{m,1}),\ \ x\mod \ann_{\ell}(\gamma_{m,2}), \dots ,\ \ x\mod \ann_{\ell}(\gamma_{m,n_m}) \right)\]
of $\coprod_{i=1}^{n_m} R/R\ann_{\ell}(\gamma_{m,i})$. Since each $\ann_{\ell}(\gamma_{m,i})$ is strongly $\Theta$-distinguished, each $R/R\ann_{\ell}(\gamma_{m,i})$ is $\Theta$-rational. Since \eqref{map 30499949ff} is monic, Remark \ref{rationals closed under quots and subs 2} gives us that $R/R\left(\cap_{i=1}^{n_m}\ann_{\ell}(\gamma_{m,i})\right)$ is also $\Theta$-rational. Since $R/R\ann_{\ell}(m)$ is a quotient of $R/R\left(\cap_{i=1}^{n_m}\ann_{\ell}(\gamma_{m,i})\right)$, Remark \ref{rationals closed under quots and subs 2} gives us the $\Theta$-rationality of $R/R\ann_{\ell}(m)$, and it also gives us that $\coprod_{m\in \homog(M)} \Sigma^{\left| m\right|} R/R\ann_{\ell}(m)$ is $\Theta$-rational. Finally, one more application of Remark \ref{rationals closed under quots and subs 2} gives us that $M$ is $\Theta$-rational, since $\epsilon_M$ is surjective.
\end{itemize}
\end{proof}

\begin{corollary}\label{rational iff dist-torsion}
Let $\Gamma$ be a coalgebra which is projective over a commutative ring $A$, and let $M$ be a graded left $\Gamma^*$-module. Then $M$ is rational if and only if $M$ is distinguished-torsion.
\end{corollary}

\begin{remark}\label{remark on h0 and H0}
Corollary \ref{rational iff dist-torsion} establishes the fundamental importance of $H^*_{\dist}$: the category of graded $\Gamma$-comodules, sitting inside the category of graded $\Gamma^*$-modules, consists precisely of those $\Gamma^*$-modules $M$ such that $H^0_{\dist}(M) \rightarrow M$ is an isomorphism. On the other hand, it is $h^*_{\dist}$, not $H^*_{\dist}$, which is defined in terms of a colimit of $\Ext$ groups. It is not clear how to calculate $H^*_{\dist}$ except in cases where it coincides with $h^*_{\dist}$. Consequently it is a matter of some importance to know, for a given coalgebra $\Gamma$, whether $h^*_{\dist}$ coincides with $H^*_{\dist}$. We have $h^*_{\dist} = H^*_{\dist}$ for all co-commutative coalgebras $\Gamma$, by the same argument as given in Remark \ref{remark on ideal sets 2}. But the motivating example for this paper is case where $\Gamma$ is the mod $p$ dual Steenrod algebra, which is not co-commutative for any prime $p$!
 
By a significantly less trivial argument, we prove in Theorem \ref{main thm on mitchell coalgebras} that a certain family of graded coalgebras, the finite-type {\em Mitchell coalgebras}, also have the property that $h^*_{\dist} = H^*_{\dist}$. The dual Steenrod algebras are finite-type Mitchell coalgebras, so we get $h^*_{\dist} = H^*_{\dist}$ in the case of the dual Steenrod algebras.
\end{remark}

\subsection{Review of (pre)torsion theories, (pre)torsion classes, and stability.}

For any graded ring $R$ and any graded left $R$-module $\Theta$, it is easy to use standard ideas to see that the functor $h^0_{\overline{\dist \Theta}}: \gr\Mod(R) \rightarrow\gr\Ab$ is left exact: since $\overline{\dist \Theta}$ is filtered, $h^0_{\overline{\dist \Theta}} = \underset{I\in \overline{\dist \Theta}}{\colim}\ \underline{\hom}_R(R/RI,-)$ is a composite of left-exact functors, namely $\underline{\hom}_R(R/RI,-)$ and $\underset{I\in\overline{\dist \Theta}}{\colim}$. We hope we do not try the reader's patience by pointing out again that it is much less obvious that $H^0_{\overline{\dist \Theta}}$ is left exact: recall that $H^0_{\overline{\dist \Theta}}(M)$ is the $R$-submodule of $M$ generated by the subgroup $h^0_{\overline{\dist \Theta}}(M)$ of $M$. This is not a very commonplace way to construct a functor, and so it is not immediately obvious what kinds of ideas might allow us to see that $H^0_{\overline{\dist \Theta}}$ is left exact. 

It turns out that the relevant ideas are those from {\em torsion theory}: it is a standard result (given below in Theorem \ref{stenstrom thm}) that the preradical associated to a hereditary pretorsion class is left exact, and in Proposition \ref{rational modules are hereditary pretorsion} we prove that the $\Theta$-rational modules are a hereditary pretorsion class, whose associated preradical---namely, $H^0_{\overline{\dist \Theta}}$---is consequently left exact. 

In this subsection, we give a ``crash course'' in the basic ideas from torsion theory that are used in our proof that $H^0_{\overline{\dist \Theta}}$ is left exact. Introductory accounts of torsion theories include sections 1.12 and 1.13 of \cite{MR1313497}, the entirety of \cite{MR880019}, \cite{MR3565424}, chapter VI of \cite{MR0389953}, and the original paper that introduced torsion theories, \cite{MR191935}. The second of those references restricts attention to the abelian category of left $R$-modules for a ring $R$, while the third deals more generally with Grothendieck categories, including graded $R$-modules, which is the desired level of generality for the applications in this paper. The fifth reference, \cite{MR191935}, assumes that the ambient abelian category is well-generated (i.e., each object has a set, rather than a proper class, of subobjects), which is also satisfied in all applications in this paper. All five are excellent references. 

Here are the relevant basics. 
\begin{definition}
Let $\mathcal{C}$ be a complete, co-complete abelian category. 
\begin{itemize}
\item
A {\em preradical on $\mathcal{C}$} is a functor $r: \mathcal{C} \rightarrow \mathcal{C}$ equipped with a natural transformation $\eta: r \rightarrow \id_{\mathcal{C}}$ such that $\eta X: rX \rightarrow X$ is a monomorphism for all objects $X$ of $\mathcal{C}$.
\item
A preradical $r$ is called a {\em radical} if $r(X/rX)$ vanishes for all objects $X$ of $\mathcal{C}$.
\item
A {\em pretorsion class in $\mathcal{C}$} is a class of objects of $\mathcal{C}$ which is closed under coproducts and quotients.
\end{itemize}
\end{definition}

\begin{definition-proposition}\label{def-prop on torsion theories}
We continue to assume that the abelian category $\mathcal{C}$ is complete and co-complete.
\begin{itemize}
\item
Suppose $\mathcal{C}$ is well-powered\footnote{A category is {\em well-powered} if, for each object $X$, the class of subobjects of $X$ (i.e., equivalence classes of monomorphisms to $X$) forms a set.}. A {\em torsion theory on $\mathcal{C}$} is a pair $(\mathcal{T},\mathcal{F})$ of full replete\footnote{A subcategory is said to be {\em replete} if it is contains every object isomorphic to one of its own objects.} subcategories of $\mathcal{C}$ such that:
\begin{enumerate}
\item if $X\in\ob\mathcal{T}$ and $Z\in\ob\mathcal{F}$, then $\hom_{\mathcal{C}}(X,Z) = 0$, and
\item if $Y\in \ob\mathcal{C}$, then there exists a short exact sequence
\[ 0 \rightarrow X \rightarrow Y \rightarrow Z \rightarrow 0\]
in $\mathcal{C}$ such that $X\in\ob\mathcal{T}$ and $Z\in\ob\mathcal{F}$.
\end{enumerate}
\item A class $\mathcal{T}$ of objects of $\mathcal{C}$ is called a {\em torsion class} if there exists a class $\mathcal{F}$ of objects of $\mathcal{C}$ such that $(\mathcal{T},\mathcal{F})$ is a torsion theory on $\mathcal{C}$. Equivalently (see Theorem 2.3 of \cite{MR191935}), a torsion class in $\mathcal{C}$ is a class of objects of $\mathcal{C}$ closed under images, coproducts, and extensions.
\item A torsion theory $(\mathcal{T},\mathcal{F})$ is called {\em hereditary} if every subobject of every object in $\mathcal{T}$ is also in $\mathcal{T}$. A pretorsion class is called {\em hereditary} if it is closed under subobjects. Consequently, a torsion class is hereditary if and only if it is the torsion class of a hereditary torsion theory.
\end{itemize}
\end{definition-proposition}

The following useful result combines Propositions 1.4, 1.7, 2.3, and 3.1 and Corollary 1.8 in chapter VI of \cite{MR0389953}:
\begin{theorem}\label{stenstrom thm}
Let $\mathcal{C}$ be a complete, co-complete abelian category. Given a preradical $r$ on $\mathcal{C}$, let $\mathcal{T}_r$ denote the collection of all objects $X$ of $\mathcal{C}$ such that $\eta X$ is an isomorphism. 
\begin{itemize}
\item If $r$ is idempotent, then $r$ and $\mathcal{T}_r$ determine one another, and consequently there is a one-to-one correspondence between idempotent preradicals on $\mathcal{C}$ and pretorsion classes in $\mathcal{C}$.
\item 
The following are equivalent:
\begin{itemize}
\item $r$ is left exact.
\item $r$ is idempotent and $\mathcal{T}_r$ is hereditary.
\end{itemize}
Consequently there is a one-to-one correspondence between left exact idempotent preradicals\footnote{Equivalently, left exact idempotent preradicals, since the left exact preradicals are all idempotent.} on $\mathcal{C}$ and hereditary pretorsion classes in $\mathcal{C}$.
\item Furthermore, $r$ is an idempotent radical if and only if $\mathcal{T}_r$ is a torsion class. Consequently we get a one-to-one correspondence between idempotent radicals on $\mathcal{C}$ and torsion theories on $\mathcal{C}$. 
\item 
Combining the above results, $r$ is a left exact radical if and only if $\mathcal{T}_r$ is a hereditary torsion class. Consequently we get a one-to-one correspondence between left exact radicals\footnote{Equivalently, left exact idempotent radicals.} on $\mathcal{C}$ and hereditary torsion theories on $\mathcal{C}$.
\end{itemize}
\end{theorem}

\subsection{The hereditary pretorsion class associated to a module $\Theta$.}

\begin{definition}
Let $R$ be a graded ring, and let $S$ be a connected ideal set in $R$. 
\begin{itemize}
\item A graded left $R$-module $M$ is {\em $S$-torsion} if the natural map $h^0_{S}(M)\hookrightarrow M$ is an isomorphism. That is, $M$ is $S$-torsion if and only if every homogeneous element of $M$ is $I$-torsion for some member $I$ of $S$.
\item A graded left $R$-module $M$ is {\em $S$-rational} if the natural map $H^0_{S}(M)\hookrightarrow M$ is an isomorphism. That is, $M$ is $S$-rational if and only if every homogeneous element of $M$ is a homogeneous $R$-linear combination of homogeneous elements of $M$, each of which is $I$-torsion for some $I\in S$.
\end{itemize}
\end{definition}
Of course every $S$-torsion module is $S$-rational, and the converse is true when $R$ is commutative. When $R$ is non-commutative, there can exist $S$-rational modules which fail to be $S$-torsion: an example is given in Examples \ref{torsion class examples 2}.

\begin{prop}\label{rational modules are hereditary pretorsion}
Let $R$ be a graded ring, and let $S$ be a connected ideal set in $R$. Then the following claims are true:
\begin{enumerate}
\item The $S$-torsion modules in $\gr\Mod(R)$ are a hereditary pretorsion class. If $h^1_S(M)$ vanishes on all $S$-torsion graded $R$-modules $M$, then the $S$-torsion modules in $\gr\Mod(R)$ are a hereditary torsion class.
\item The $S$-rational modules in $\gr\Mod(R)$ are a pretorsion class. 
\item If we furthermore have that $S = \overline{\dist \Theta}$ for some graded left $R$-module $\Theta$, then the $S$-rational modules in $\gr\Mod(R)$ are a hereditary pretorsion class. 
\end{enumerate}
\end{prop}
\begin{proof}
First, from Definition \ref{def of distinguished}, we know that $H^0_{S}$ is a preradical\footnote{Note that, since $h^0_{S}$ does not take values in $\gr\Mod(R)$ but only in $\gr\Ab$, $h^0_{S}$ is not a preradical.} on $\gr\Mod(R)$.
\begin{enumerate}
\item  
By Theorem \ref{stenstrom thm}, to know that the $\Theta$-torsion modules form a hereditary pretorsion class, we need to show that the $S$-torsion modules are closed under coproducts, quotients, and submodules.
We begin with coproducts. Given a graded left $R$-module $M$, let $\eta_M: h^0_S(M)\rightarrow M$ denote the natural monomorphism.
Then, given a set $\{ M_j: j\in J\}$ of $S$-torsion modules in $\gr\Mod(R)$, we have natural maps fitting into a commutative diagram
\begin{equation}\label{diag 0340405} \xymatrix{
 \underset{I\in S}{\colim}\ \underline{\hom}_R\left(R/RI,\underset{j}{\coprod} M_j \right) \ar[r]^(.6){\eta_{\underset{j}{\coprod} M_j}} 
  & \underset{j}{\coprod} M_j \\
 \underset{i\in S}{\colim}\ \underset{j}{\coprod} \underline{\hom}_R(R/RI,M_j) \ar[u] 
  & \underset{j}{\coprod} \underset{I\in S}{\colim}\ \underline{\hom}_R(R/RI,M_j) \ar[l]^{\cong} \ar[u]^{\cong}_{\underset{j}{\coprod} \eta_{M_j}} .
}\end{equation}
Since $\underset{j}{\coprod} \eta_{M_j}$ is an isomorphism and hence epic, the last map in the composite depicted in diagram \ref{diag 0340405} must also be epic. That map is $\eta_{\underset{j}{\coprod} M_j}$, which is also monic since $S$ is connected. Hence $\eta_{\underset{j}{\coprod} M_j}$ is an isomorphism, i.e., $\underset{j}{\coprod}M_j$ is $S$-torsion.

We also need to show that the $S$-torsion modules are closed under quotients. This is a consequence of an easy diagram chase along the lines of the argument just given for closure under coproducts. Consequently the $S$-torsion modules form a pretorsion class. A similar diagram chase suffices to prove that, if $h^1_S(M)$ vanishes for all $S$-torsion $M$, then the $S$-torsion modules are closed under extensions in $\gr\Mod(R)$, and consequently form a torsion class.

Finally, we also need to show that the $S$-torsion modules are closed under submodules. This follows easily from the understanding of the $S$-torsion modules as those graded $R$-modules in which every homogeneous element is $I$-torsion for some $I\in S$.
\item 
We need to show that the $S$-rational modules are closed under coproducts. Given a set $\{ M_j: j\in J\}$ of $S$-rational modules in $\gr\Mod(R)$, any given homogeneous element $m$ of the coproduct $\coprod_{j\in J}M_j$ is a homogeneous $R$-linear combination of elements which are each concentrated in a single summand $M_j$ of $\coprod_{j\in J}M_j$. Each of those elements, in turn, is a homogeneous $R$-linear combination of homogeneous elements which are $I$-torsion for various ideals $I$ of $S$. Consequently $m$ is a homogeneous $R$-linear combination of homogeneous elements which are $I$-torsion for various ideals $I$ of $S$. Consequently $\coprod_{j\in J} M_j$ is $S$-rational.

Showing that the $S$-rational graded $R$-modules are closed under quotients is a simple matter of a diagram chase. Consequently the $S$-rational modules form a pretorsion class in graded $R$-modules. 
\item Given a graded left $R$-module $\Theta$, it is a consequence of Theorem \ref{theta-rationality and torsion} that the $\overline{\dist \Theta}$-rationality---that is, $\Theta$-rationality---of a graded $R$-module $M$ is inherited by all graded submodules of $M$.
\end{enumerate}
\end{proof}
\begin{corollary}\label{H0 is left exact}
Let $R$ be a graded ring, and let $\Theta$ be a graded $R$-module. Then $H^0_{\overline{\dist \Theta}}: \gr\Mod(R) \rightarrow \gr\Mod(R)$ is left exact.
\end{corollary}
\begin{proof}
Immediate consequence of Theorem \ref{stenstrom thm}
and the third part of Proposition \ref{rational modules are hereditary pretorsion}.
\end{proof}

\begin{corollary}\label{H0 rat and iota}
Let $A$ be a commutative ring, let $\Gamma$ be an $A$-coalgebra projective over $A$, and let $R = \Gamma^*$ be the $A$-linear dual algebra of $\Gamma$. Then $H^0_{\dist}: \gr\Mod(R)\rightarrow \gr\Mod(R)$ is naturally isomorphic to the composite $\iota\circ\tr$ of the rational functor $\tr: \gr\Mod(\Gamma^*)\rightarrow \gr\Comod(\Gamma)$ with the inclusion $\iota: \gr\Comod(\Gamma)\rightarrow\gr\Mod(\Gamma^*)$ of the comodules into the modules.
\end{corollary}

Here are some examples to demonstrate the importance of the hypotheses involved in Proposition \ref{rational modules are hereditary pretorsion}.
\begin{example}\label{torsion class examples}
There exist rings $R$ and filtered (hence connected) ideal sets $S$ in $R$ such that the pretorsion class of $S$-torsion modules fails to be a torsion class, and also such that the pretorsion class of $S$-rational modules fails to be a torsion class. Here is a simple example of both: let $R = \mathbb{Z}/4\mathbb{Z}$, and let $S$ be $\{ (2)\}$. Then $h^0_S(M) = H^0_S(M)$ is the $2$-torsion submodule of a $\mathbb{Z}/4\mathbb{Z}$-module $M$. We have the short exact sequence of $\mathbb{Z}/4\mathbb{Z}$-modules
\[ 
 0 
  \rightarrow \mathbb{Z}/2\mathbb{Z} 
  \rightarrow \mathbb{Z}/4\mathbb{Z} 
  \rightarrow \mathbb{Z}/2\mathbb{Z} 
  \rightarrow 0\]
in which the left- and right-hand nonzero modules are $S$-torsion and also $S$-rational, but the middle module, $\mathbb{Z}/4\mathbb{Z}$, is neither $S$-torsion nor $S$-rational.
\end{example}
\begin{example}\label{torsion class examples 2}
When the connected ideal set $S$ is not equivalent to $\overline{\dist \Theta}$ for some graded left $R$-module $\Theta$, it is not necessarily the case that the pretorsion class consisting of the $S$-rational modules is hereditary. Here is an example: let the ring $R$ be the subalgebra $A(1)$ of the $2$-primary Steenrod algebra generated by $\Sq^1$ and $\Sq^2$. Let $S$ be $\{ A(1)\Sq^1\}$, that is, $S$ has just one member, and that member is the left ideal of $A(1)$ generated by $\Sq^1$. 

Let $M$ be the graded left $A(1)$-submodule of $A(1)$ generated by the element $\Sq^1$. For clarity, write $x$ for this generator, so that $M$ is a four-dimensional $\mathbb{F}_2$-vector space with homogeneous $\mathbb{F}_2$-linear basis \[ \{ x,\ \Sq^2x,\ \Sq^1\Sq^2x,\ \Sq^2\Sq^1\Sq^2x\}.\] The action of the ideal $A(1)\Sq^1$ on the generator $x\in M$ is trivial, since $\Sq^1x=0$. However, the action of $A(1)\Sq^1$ on $M$ is not trivial: the element $\Sq^1\in A(1)\Sq^1$ acts nontrivially on $\Sq^2x$. 

The upshot is that $H^0_S(M) = M$, but $h^0_S(M)$ is strictly smaller than $M$. In particular, $h^0_S(M)$ is the $\mathbb{F}_2$-linear span of $x$, $\Sq^1\Sq^2x,$ and $\Sq^2\Sq^1\Sq^2x$. Hence $M$ is $S$-rational but not $S$-torsion. 

Furthermore, let $M^{\prime}$ be the $A(1)$-submodule of $M$ generated by $\Sq^2x$. Since $h^0_S(M^{\prime})$ is $\mathbb{F}_2$-linearly spanned by $\Sq^1\Sq^2x$ and $\Sq^2\Sq^1\Sq^2x$, we have that $h^0_S(M^{\prime}) = H^0_S(M^{\prime}) \neq M^{\prime}$, so $M^{\prime}$ is neither $S$-rational nor $S$-torsion, despite being a submodule of a $S$-rational module. Hence the pretorsion class of $S$-rational $A(1)$-modules is not hereditary.
\end{example}

\section{Distinguished local cohomology of products.}
\label{Distinguished local cohomology...}
\subsection{Derived products of comodules are given by distinguished local cohomology.}



\begin{lemma}\label{bdd above products prelemma}
Let $A$ be a field, and let $\Gamma$ be a finite-type graded $A$-module concentrated in nonpositive degrees. Let $S$ be a set, and let $\{ M_s: s\in S\}$ be a uniformly bounded-above set of graded $A$-modules. Then the natural map of graded $A$-modules \[\Gamma\otimes_A \prod_{s\in S} M_s \rightarrow \prod_{s\in S} \Gamma\otimes_A M_s \] is an isomorphism.
\end{lemma}
\begin{proof}
Elementary; see for example Lemma 2.4 in \cite{salchenoughproj}. (There the result is stated and proven for nonnegative gradings, rather than nonpositive gradings, but the nonpositively-graded case is proven in exactly the same manner.)
\end{proof}

\begin{lemma}\label{bdd above products lemma}
Let $A$ be a field, and let $\Gamma$ be a finite-type graded $A$-coalgebra concentrated in nonpositive degrees. Let $S$ be a set, and let $d: S \rightarrow \mathbb{Z}$ be a function whose set of values $\{ d(s): s\in S\}$ is bounded above. Then the graded $\Gamma$-comodule $\prod^{\Gamma}_{s\in S} \Sigma^{d(s)} \Gamma$ is the Cartesian product $\prod_{s\in S}\Sigma^{d(s)}\Gamma$.

Put more precisely: if we write $G$ for the forgetful functor $\gr\Comod(\Gamma)\rightarrow\gr\Mod(A)$, then under the stated hypotheses, the natural map of graded $A$-modules
\[ G\prod^{\Gamma}_{s\in S} \Sigma^{d(s)} \Gamma \rightarrow \prod_{s\in S} \Sigma^{d(s)}G\Gamma\]
is an isomorphism.
\end{lemma}
\begin{proof}
The underlying graded $A$-module $G\prod^{\Gamma}_{s\in S} \Sigma^{d(s)}\Gamma$ of $\prod^{\Gamma}_{s\in S} \Sigma^{d(s)}\Gamma$ is the graded $A$-module pullback
\begin{equation}\label{pb square 1}\xymatrix{
 G\prod^{\Gamma}_{s\in S} \Sigma^{d(s)}\Gamma
  \ar[rr]\ar[d] && \Gamma\otimes_A \prod_{s\in S}\Sigma^{d(s)}\Gamma \ar[d] \\
 \prod_{s\in S} \Sigma^{d(s)}\Gamma \ar[rr]_(.45){\prod_{s\in S} \Sigma^{d(s)}\Delta} && \prod_{s\in S} \Sigma^{d(s)} \Gamma\otimes_A\Gamma.
}\end{equation}
The right-hand vertical map in \eqref{pb square 1} is the natural comparison map, which is an isomorphism by Lemma \ref{bdd above products prelemma}. Now the claim follows, since the pullback of an isomorphism is an isomorphism.
\end{proof}

\begin{theorem}\label{main thm 1}
Suppose that $\Gamma$ is a graded $A$-coalgebra which is projective as an $A$-module.
Then the following claim is true:
\begin{enumerate}
\item For each nonnegative integer $n$ and each graded left $\Gamma^*$-module $M$, we have an isomorphism $H^n_{\dist}(M) \cong \iota\left( R^n\tr(M)\right)$, natural in the variable $M$.
\end{enumerate}
Suppose furthermore that $A$ is a field, and that $\Gamma^*$ is finite-type and connected. Then the following claims are also each true:
\begin{enumerate}
\item[(2)] If $M$ is a bounded-above injective graded $\Gamma$-comodule, then the graded $\Gamma^*$-module $\iota(M)$ is injective. 
\item[(3)] The rational graded $\Gamma^*$-modules are a hereditary 
pretorsion class\footnote{We do {\em not} claim that this pretorsion class is a torsion class. Indeed, $\iota$ can fail to preserve injectivity, and rational graded modules which are not bounded above can fail to embed into any injective rational graded modules. We take up these issues in the preprint \cite{selfinjectivitypreprint}.} in $\gr\Mod(\Gamma^*)$. 
\item[(4)] If $M$ is a bounded-above graded $\Gamma^*$-module, then the distinguished local cohomology groups $H^n_{\dist}(M)$ vanish for all $n>0$. 
\item[(5)] Every bounded-above graded $\Gamma^*$-module is rational.
\item[(6)] Let $I$ be a set, and suppose that, for each $i\in I$, we have a bounded-above\footnote{To be clear: each comodule $M_i$ is assumed to be bounded above, but we do {\em not} assume that the whole collection of comodules $\{ M_i: i\in I\}$ is {\em uniformly} bounded above!} graded $\Gamma$-comodule $M_i$. Then, for each nonnegative integer $n$, the $n$th distinguished local cohomology module $H^n_{\dist}\left(\prod_{i\in I} \iota(M_i)\right)$ of the product of the graded $\Gamma^*$-modules $\iota(M_i)$ is isomorphic to $\iota$ applied to the $n$th right derived functor $R^n\prod^{\Gamma}_{i\in I}\{ M_i\}$ of the product functor $\gr\Comod(\Gamma)^I\stackrel{\prod^{\Gamma}}{\longrightarrow} \gr\Comod(\Gamma)$. That is, we have an isomorphism
\[ \iota\left( R^{n}\prod^{\Gamma}_i\left(\left\{ M_i:i\in I\right\}\right)\right)
 \cong H^n_{\dist}\left( \prod_{i\in I}\iota( M_i)\right).\]
\end{enumerate}
\end{theorem}
\begin{proof}
\begin{enumerate}
\item Since $\tr$ has an exact left adjoint (namely, $\iota$), $\tr$ sends injectives to injectives. Consequently we have a Grothendieck spectral sequence $R^s\iota \left(R^t\tr(M)\right) \Rightarrow R^{s+t}(\iota\circ\tr)(M)$, and it collapses to the $s=0$ line since $\iota$ is exact. We consequently have isomorphisms $\iota \left( R^t\tr(M)\right)\cong R^t(\iota\circ \tr)(M)\cong H^t_{\dist}(M)$ due to Corollary \ref{H0 rat and iota}.
\item The injective graded $\Gamma$-comodules are the retracts of the extended graded $\Gamma$-comodules. Since $A$ is a field, the extended graded $\Gamma$-comodules are the coproducts of suspensions of $\Gamma$ itself. Consequently, if we can show that $\iota(M)$ is injective when $M$ is a bounded-above coproduct of suspensions of $\Gamma$, then $\iota(M)$ must be injective for all bounded-above injective $\Gamma$-comodules $M$.

To that end, we suppose that $M$ is the graded $\Gamma$-comodule $\coprod_{s\in S} \Sigma^{d(s)} \Gamma$, where $S$ is a set, and where $d: S\rightarrow\mathbb{Z}$ is a function whose image in $\mathbb{Z}$ is bounded above. Then $M$ is the extended graded $\Gamma$-comodule on the graded $A$-vector space $\coprod_{s\in S} \Sigma^{d(s)}A$. The natural map of graded $A$-vector spaces $\coprod_{s\in S} \Sigma^{d(s)}A \rightarrow \prod_{s\in S} \Sigma^{d(s)}A$ is a split monomorphism, so upon applying the extended comodule functor $E$, we have that the map of graded $\Gamma$-comodules
\[ \coprod_{s\in S} \Sigma^{d(s)} \Gamma \cong E\left( \coprod_{s\in S} \Sigma^{d(s)}A\right) \rightarrow E\left(\prod_{s\in S} \Sigma^{d(s)}A\right) \cong \prod^{\Gamma}_{s\in S} \Sigma^{d(s)} \Gamma \]
is also a split monomorphism. 
Applying $\iota$ then yields a split monomorphism of graded $\Gamma^*$-modules
\begin{equation}\label{split mono 30} \iota M \cong \coprod_{s\in S} \Sigma^{d(s)} \iota\Gamma  \rightarrow \iota\prod^{\Gamma}_{s\in S} \Sigma^{d(s)} \Gamma. \end{equation}
If we knew that the canonical comparison map 
\begin{equation}\label{comparison map 31}\iota\prod^{\Gamma}_{s\in S} \Sigma^{d(s)} \Gamma \rightarrow \prod_{s\in S} \Sigma^{d(s)}\iota\Gamma,\end{equation} were an isomorphism, then \eqref{split mono 30} would exhibit $\iota M$ as a summand in a product of injective graded $\Gamma^*$-modules, hence $\iota M$ would be injective. By the assumption that the degrees $\{ d(s): s\in S\}$ are bounded above and by Lemma \ref{bdd above products lemma}, the map \eqref{comparison map 31} is indeed an isomorphism, and we are done\footnote{Without the bounded-aboveness assumption on the degrees $\{ d(s): s\in S\}$, the map \eqref{comparison map 31} is not always an isomorphism. Indeed, if \eqref{comparison map 31} were always an isomorphism regardless of any degree bounds, then by the argument just given, $\iota$ would send every injective comodule to an injective module. In the preprint \cite{selfinjectivitypreprint} one can find a proof that $\iota$ unfortunately does not have that desirable property. 

The same preprint also contains a completely different proof of claim 2 of this theorem, i.e., the claim that $\iota$ sends bounded-above injective graded comodules to injective graded modules.}.
\item[(3)] Corollary \ref{rational iff dist-torsion} together with Proposition \ref{rational modules are hereditary pretorsion} establishes that the rational modules form a hereditary pretorsion class.
\item[(4, part 1)] We first prove that $H^n_{\dist}(M)$ vanishes for $n>0$ for all bounded-above graded {\em rational} $\Gamma^*$-modules $M$. We will then return and finish the proof of claim (4), after proving (5), in order to lift the rationality assumption on $M$.

Let $M$ be a bounded-above rational graded $\Gamma^*$-module. Then the extended graded $\Gamma$-comodule on the underlying graded $\mathbb{F}_p$-vector space of $\tr(M)$ is also bounded-above. Hence $\tr(M)$ admits a resolution by bounded-above graded-injective $\Gamma$-comodules. Since we have already shown that $\iota$ sends bounded-above injectives to bounded-above injectives, $M$ admits a resolution $\mathcal{I}^{\bullet}$ by rational, graded-injective $\Gamma^*$-modules. Consequently $(\iota\circ\tr)(\mathcal{I}^{\bullet})\simeq \mathcal{I}^{\bullet}$ is acyclic.

Consider again the same Grothendieck spectral sequence
\begin{align*}
 (R^s\iota\circ R^t\tr)(M) &\Rightarrow R^{s+t}(\iota\circ\tr)(M)
\end{align*}
from the proof of claim (1) of this theorem.
Its $E_2^{s,t}$-term $(R^s\iota\circ R^t\tr)(M)$ is isomorphic to $\iota(R^t\tr(M))$, since $\iota$ is exact, so the spectral sequence collapses to the $s=0$-line with no differentials. Since $M$ is rational, the spectral sequence's abutment $R^{*}(\iota\circ\tr)(M)$ vanishes in degrees $*>0$, by the acyclicity argument in the previous paragraph. Hence $\iota R^t\tr(M)$, i.e., $H^t_{\dist}(M)$, vanishes for $t>0$.
\item[(5)]
Suppose that $M$ is a graded $\Gamma^*$-module. For each integer $n$, write $\conn_n(M)$ for the graded sub-$\Gamma^*$-module of $M$ generated by all homogeneous elements of degree $\geq n$. Then we have the sequence of monomorphisms \begin{equation}\label{filt 13}\dots \hookrightarrow \conn_n(M) \hookrightarrow \conn_{n-1}(M) \hookrightarrow \conn_{n-2}(M) \hookrightarrow\dots\end{equation}
of graded $\Gamma^*$-modules, and its colimit is $M$. If $M$ is bounded-above, then each of the submodules $\conn_n(M)$ is bounded (i.e., both bounded-above and bounded-below). 

Consequently every bounded-above graded $\Gamma^*$-module is a colimit of bounded graded $\Gamma^*$-modules.
By Proposition \ref{rationals closed under quots and subs}, the rational $\Gamma^*$-modules are closed under cokernels and coproducts in $\gr\Mod(\Gamma^*)$, hence closed under all small colimits in $\gr\Mod(\Gamma^*)$. Hence, if we can show that the bounded graded $\Gamma^*$-modules are rational, then all the bounded-above graded $\Gamma^*$-modules will be rational.

So we suppose that $N$ is a bounded graded $\Gamma^*$-module. We carry out an induction to prove that $H^n_{\dist}(N)$ vanishes for $n>0$. The initial step in the induction is the case in which $N$ is concentrated in a single degree. Since $\Gamma^*$ is connected and $A$ is a field, we have that $N$ splits (as a $\Gamma^*$-module) as a coproduct of copies of $\Gamma^*/\conn_1(\Gamma^*) \cong A$. 
It is straightforward to see that the graded $\Gamma^*$-module $\Gamma^*/\conn_1(\Gamma^*)$ must be rational. By Proposition \ref{rationals closed under quots and subs}, a coproduct of rational modules is rational, so $N$ is rational, as desired. This completes the initial step in the induction.

The inductive step is as follows: suppose $m$ is a nonnegative integer, and suppose that we have already shown that $N^{\prime}$ is rational for all graded $\Gamma^*$-modules $N^{\prime}$ which are trivial except in $\leq m$ consecutive degrees. (That is, the inductive hypothesis is that, if $r$ is an integer and $N^{\prime}$ is a graded $\Gamma^*$-module which is trivial except in grading degrees $r, r+1, \dots ,r+m-1$, then $N^{\prime}$ is rational.) 
Suppose that $N$ is a graded $\Gamma^*$-module which is trivial except in $m+1$ consecutive degrees. Let $r$ be the lowest degree in which $N$ is nontrivial. Let $N^{\geq r+1}$ be the graded $A$-submodule of $N$ generated by all homogeneous elements of degree $\geq r+1$.
Then $N^{\geq r+1}$ is a graded $\Gamma^*$-submodule of $N$, since $\Gamma^*$ is connective. We consequently have a short exact sequence of graded $\Gamma^*$-modules
\begin{equation}\label{ses 04040} 0 \rightarrow N^{\geq r+1} \rightarrow N \rightarrow N/N^{\geq r+1} \rightarrow 0\end{equation}
and $N^{\geq r+1}$ is trivial except in $\leq m$ consecutive degrees, hence is rational by the inductive hypothesis, and consequently $H^n_{\dist}(N^{\geq r+1})$ vanishes for all $n>0$, by the previous part of this theorem.
Consequently (and using the left-exactness of $H^0_{\dist}$, proven in Corollary \ref{H0 is left exact}), applying $H^*_{\dist}$ to \eqref{ses 04040} yields exactness of the top row in the commutative diagram
\begin{equation}\label{comm diag 0595959}
\xymatrix{
 0 \ar[r] \ar[d] &
  H^0_{\dist}(N^{\geq r+1}) \ar[r]\ar[d]^{\cong} &
  H^0_{\dist}(N) \ar[r]\ar[d] &
  H^0_{\dist}(N/N^{\geq r+1}) \ar[r] \ar[d]^{\cong} &
  0 \ar[d] \\
 0 \ar[r] &
  N^{\geq r+1} \ar[r] &
  N \ar[r] &
  N/N^{\geq r+1} \ar[r] &
 0.}\end{equation}
The vertical maps indicated with the symbol $\cong$ in diagram \eqref{comm diag 0595959} are isomorphisms by the initial step (for $N/N^{\geq r+1}$, since it is concentrated in a single degree), and by the inductive hypothesis (for $N^{\geq r+1}$). 
So $H^0_{\dist}(N)\rightarrow N$ is also an isomorphism, i.e., $N$ is also rational, completing the inductive step. So every bounded graded $\Gamma^*$-module is rational, as desired.
\item[(4, part 2)] Now it is easy to finish the proof of claim (4): we have already shown that $H^n_{\dist}(M)$ vanishes for all bounded-above graded rational $\Gamma^*$-modules $M$, and we have just shown that every bounded-above $\Gamma^*$-module is rational. So $H^n_{\dist}$ vanishes for $n>0$ on any bounded-above $\Gamma^*$-module.
\item[(6)]
Since products of injectives are injective, the functor $\prod_I:\gr\Mod(\Gamma^*)^I \rightarrow\gr\Mod(\Gamma^*)$ preserves injectives. We have shown that $\iota$ also preserves bounded-above injectives. It is classical and straightforward that an object of the functor category $\gr\Mod(\Gamma^*)^I$ is injective if and only if it is objectwise injective. So, if $M_i$ is a bounded-above graded-injective $\Gamma$-comodule for each $i\in I$, then the product $\prod_{i\in I} \iota(M_i)$ is a graded-injective $\Gamma^*$-module.
So we get a Grothendieck spectral sequence
\begin{align*}
 E_2^{s,t} \cong R^s\tr R^t\left(\prod_{i\in I}\circ\iota\right)\left(\left\{ M_i:i\in I\right\}\right) &\Rightarrow R^{s+t}\left( \tr\circ\prod_{i\in I}\circ\iota\right)\left(\left\{ M_i:i\in I\right\}\right) 
\\ &\cong R^{s+t}\left( \prod^{\Gamma}_i\circ \tr\circ\iota\right)\left(\left\{ M_i:i\in I\right\}\right) 
\\ &\cong R^{s+t} \prod^{\Gamma}_i\left(\left\{ M_i:i\in I\right\}\right).
\end{align*}
This spectral sequence collapses to the $t=0$ line at its $E_2$-page, since $\iota$ and $\prod_I$ are each exact and so their composite $\prod_I\circ\iota$ is exact. 
Consequently we have isomorphisms
\begin{align}
\nonumber \iota R^{s}\prod^{\Gamma}_i\left(\left\{ M_i:i\in I\right\}\right)
  &\cong \iota R^s\tr \left(\prod_{i\in I}\iota( M_i)\right) \\
\label{iso 00094912}  &\cong H^s_{dist}\left( \prod_{i\in I}\iota( M_i)\right) ,
\end{align}
with isomorphism \eqref{iso 00094912} due to the first part of this theorem.
\end{enumerate}
\end{proof}

Among other things, Theorem \ref{main thm 1} proves that the higher distinguished local cohomology groups vanish on the bounded-above graded modules which come (via $\iota$) from comodules.
One might try to think of this result as telling us that distinguished local cohomology $H^*_{\dist}(M)$ is a cohomology theory that tells us how far the module $M$ is from being a comodule. This requires a bit of care. For example, the higher distinguished local cohomology groups {\em also} vanish on some $\Gamma^*$-modules which {\em don't} come from comodules. In Theorem 12 of section 13.3 of \cite{MR738973}, Margolis proves that the Steenrod algebras are self-injective, and more generally, that if $\Gamma^*$ is a $\mathcal{P}$-algebra in the sense of Margolis, then $\Gamma^*$ is self-injective. Letting $\Gamma^*$ be the Steenrod algebra, all the hypotheses of Theorem \ref{main thm 1} are satisfied, but not only does $H^n_{\dist}(\Gamma^*)$ vanish for all $n>0$, we also have that $H^0_{\dist}(\Gamma^*)$ vanishes. 
So $\Gamma^*$ cannot be in the image of $\iota$, since $0\cong H^0_{\dist}(\Gamma^*) \cong \iota(\tr(\Gamma^*))$.

So while one can truthfully say that the higher $H^*_{\dist}$ groups ``detect the failure of a bounded-above $\Gamma^*$-module to be a $\Gamma$-comodule,'' it is not the case that the $\Gamma$-comodules are {\em precisely} those $\Gamma^*$-modules on which the higher $H^*_{\dist}$ groups vanish. Still, there are satisfying structural relationships between the category of $\Gamma$-comodules and the category of $\Gamma^*$-modules, such as the following theorem, which establishes that every graded $\Gamma^*$-module is an extension of an $n$-co-connected rational graded $\Gamma^*$-module by an $n$-connective graded $\Gamma^*$-module. Recall that a graded abelian group is said to be {\em $n$-co-connected} if it is trivial in degrees $\geq n$, and {\em $n$-connective} if it is trivial in degrees $<n$.
\begin{theorem}\label{structure thm}
Suppose that $A$ is a field, and that $\Gamma$ is a graded $A$-coalgebra such that the dual algebra $\Gamma^*$ is finite-type and connected. Let $n$ be an integer. Then, for each graded $\Gamma^*$-module $M$, we have 
a short exact sequence of graded $\Gamma^*$-modules
\begin{equation}\label{ses 40909090} 0 \rightarrow \conn_n(M) \rightarrow M \rightarrow \iota(\comod_n(M)) \rightarrow 0,\end{equation}
where $\conn_n(M)$ is a {\em $n$-connective} graded $\Gamma^*$-module, and $\comod_n(M)$ is an {\em $n$-co-connected} graded $\Gamma$-comodule. This sequence is natural in the variable $M$.

Furthermore, the higher distinguished local cohomology of $M$ depends only on $\conn_n(M)$. That is, $H^i_{\dist}(\conn_n(M)) \rightarrow H^i_{\dist}(M)$ is an isomorphism for all $n$ and for all $i>0$.
\end{theorem}
\begin{proof}
Let $\conn_n(M)$ simply be the graded $\Gamma^*$-submodule of $M$ generated by all elements in degrees $\geq n$. The quotient $M/\conn_n(M)$ is then $n$-co-connected, hence bounded above, hence is $\iota$ of a graded $\Gamma$-comodule by Theorem \ref{main thm 1}. 
\end{proof}
It is easy to see that the sequence \eqref{ses 40909090} is natural in the integer $n$, in the sense that we have a commutative diagram of graded $\Gamma^*$-modules
\begin{equation}\label{comm diag 40990909}\xymatrix{
\vdots \ar[d] &
 \vdots \ar@{^{(}->}[d] &
 \vdots \ar[d] &
 \vdots \ar@{->>}[d] &
 \vdots \ar[d] \\
0 \ar[r] \ar[d] & 
 \conn_{n+1}(M) \ar[r]\ar@{^{(}->}[d] & 
 M \ar[r]\ar[d]^{\id} & 
 \iota(\comod_{n+1}(M))\ar[r]\ar@{->>}[d] & 0 \ar[d] \\
0 \ar[r] \ar[d] & 
 \conn_n(M) \ar[r]\ar@{^{(}->}[d] & 
 M \ar[r]\ar[d]^{\id} & 
 \iota(\comod_n(M))\ar[r]\ar@{->>}[d] & 0 \ar[d] \\
0 \ar[r] \ar[d] & 
 \conn_{n-1}(M) \ar[r]\ar@{^{(}->}[d] & 
 M \ar[r]\ar[d]^{\id} & 
 \iota(\comod_{n-1}(M))\ar[r]\ar@{->>}[d] & 0 \ar[d] \\
\vdots &
 \vdots &
 \vdots &
 \vdots &
 \vdots 
}\end{equation}
with exact rows.
Taking the limit of each column in \eqref{comm diag 40990909} yields:
\begin{corollary}\label{modules are limits of comodules}
Suppose that $A$ is a field, and that $\Gamma^*$ satisfies the hypotheses stated in Theorem \ref{structure thm}. Then every graded $\Gamma^*$-module is the limit of a Mittag-Leffler sequence of rational graded $\Gamma^*$-modules. 
\end{corollary}
\begin{proof}
The sequence of monomorphisms (i.e., the left-hand nonzero column) in \eqref{comm diag 40990909} is eventually constant in each grading degree, hence is a Mittag-Leffler sequence in each grading degree. Consequently $R^1\underset{n\rightarrow\infty}{\lim}\left(\conn_n(M)\right)$ vanishes in the category $\gr\Mod(\Gamma^*)$. (This argument does not show that $R^1\underset{n\rightarrow\infty}{\lim}\left( \conn_n(M)\right)$ vanishes in the ungraded category $\Mod(\Gamma^*)$: the grading is important here.) Consequently we have the isomorphism $M\stackrel{\cong}{\longrightarrow}\underset{n\rightarrow\infty}{\lim}\iota(\comod_n(M))$ in $\gr\Mod(\Gamma^*)$.
\end{proof}
It would be nice to have a rigorous way to interpret Theorem \ref{structure thm} as stating that ``the category of graded $\Gamma^*$-modules is an extension of the category of co-connected $\Gamma$-comodules by the category of connective $\Gamma^*$-modules,'' but there does not seem to be a notion of ``extension of abelian categories'' in the literature which is of the right kind of generality to include the situation of Theorem \ref{structure thm} as an example. Marmaridis's notion of ``extensions of abelian categories,'' from \cite{MR1213784}, does not suffice, since $\gr\Mod(\Gamma^*)$ is not monadic or comonadic over connective $\Gamma^*$-modules or over co-connected $\Gamma$-comodules (see Theorem 2.6 of \cite{MR1780016} for the relationship between (co)monadicity and Marmaridis's extension theory). The situation of Theorem \ref{structure thm} is not a special case of a ``deformation of abelian categories'' in the sense of \cite{MR2238922} either.

\begin{corollary}\label{uniqueness of module cat}
Let $A,\Gamma^*$ be as in Corollary \ref{modules are limits of comodules}. Then the only full subcategory of $\gr\Mod(\Gamma^*)$ which contains the rational $\Gamma^*$-modules and which is closed under kernels and countable products is $\gr\Mod(\Gamma^*)$ itself.
\end{corollary}
\begin{proof}
Sequential limits in $\gr\Mod(\Gamma^*)$ are kernels of maps between countable products, so this follows from Corollary \ref{modules are limits of comodules}.
\end{proof}

\section{Mitchell coalgebras.}
\label{Mitchell coalgebras...}

Throughout this section, and for the rest of this paper, we assume that the ground ring $A$ is a field, so that we may freely use Theorem \ref{main thm 1}.

A technically complicated, but nevertheless important, example of a filtered ideal set arises from a coalgebra satisfying the {\em Mitchell condition,} which we define in Definition \ref{def of mitchell condition}. It requires that we first recall (from chapter 13 of \cite{MR738973}) the definition of a  ``$\mathcal{P}$-algebra'':
\begin{definition}\label{def of p-alg}
A {\em $\mathcal{P}$-algebra} is a union of a sequence of subalgebras $B(0) \subsetneq B(1) \subsetneq \dots$ such that each $B(n)$ is a Poincar\'{e} algebra, and each $B(n+1)$ is flat over $B(n)$. Here a ``Poincar\'{e} algebra,'' as in \cite{MR335572}, is a finite-dimensional graded connected 
algebra $A$ over a field $k$ such that there exists a map of graded $k$-modules $e: A\rightarrow \Sigma^{-n} k$, for some integer $n$, such that the pairing $A^q\otimes_k A^{n-q}\stackrel{\nabla}{\longrightarrow} A^n \stackrel{\Sigma^n e}{\longrightarrow} k$ is nonsingular.
\end{definition}
Of course the most important examples of $\mathcal{P}$-algebras are the Steenrod algebras: the mod $p$ Steenrod algebra is a $\mathcal{P}$-algebra for every prime $p$, by Proposition 7 from section 15.1 of \cite{MR738973}.

\begin{definition}\label{def of mitchell condition}
Suppose $\Gamma$ is a graded coalgebra over a field $A$. 
By a {\em Mitchell decomposition of $\Gamma$} we mean the following data:
\begin{itemize} 
\item a sequence $\dots \rightarrow \Gamma(2)\rightarrow \Gamma(1) \rightarrow\Gamma(0)$ of surjective graded $A$-coalgebra morphisms, with each $\Gamma(n)$ a graded quotient $A$-coalgebra of $\Gamma$, 
such that each of the dual algebras $\Gamma^*(n)$ is a Poincar\'{e} algebra, and such that the left $\Gamma^*(n)$-action on $\Gamma^*(n+1)$ arising from the dual map $\Gamma^*(n)\rightarrow\Gamma^*(n+1)$ of $\Gamma(n+1)\rightarrow\Gamma(n)$ makes $\Gamma^*(n+1)$ flat over $\Gamma^*(n)$.
\item For each nonnegative integer $n$, a homogeneous element $\omega_n \in \Gamma(n)$ whose associated map of graded $A$-modules $e:\Gamma^*(n)\rightarrow \Sigma^{-\left| \omega_n\right|} A$ has the property that the pairing 
\begin{equation}\label{pairing 0439} \Gamma^*(n)^q\otimes_A \Gamma^*(n)^{\left| \omega_n\right|-q}\stackrel{\nabla}{\longrightarrow} \Gamma^*(n)^{\left| \omega_n\right|} \stackrel{\Sigma^{\left| \omega_n\right|} e}{\longrightarrow} A\end{equation} is nonsingular.
\item For each nonnegative integer $n$, an extension of the natural graded left $\Gamma^*(n)$-module structure of $\Gamma^*(n)$ to a graded left action of $\Gamma^*$ on $\Gamma^*(n)$, together with graded left $\Gamma^*(n)$-module homomorphisms $\sigma_n: \Gamma^*\rightarrow \Gamma^*(n)$ and $\sigma_{n,n+1}: \Gamma^*(n+1) \rightarrow \Gamma^*(n)$ such that 
\begin{enumerate}
\item the composite of the $A$-algebra injection $\Gamma^*(n)\hookrightarrow \Gamma^*$ with the left $A$-module morphism $\sigma_n: \Gamma^* \rightarrow\Gamma^*(n)$ is the identity on $\Gamma^*(n)$,
\item the duality isomorphism $\Gamma^*(n)\stackrel{\cong}{\longrightarrow}\Sigma^{\left| \omega_n\right|}\Gamma^*(n)^*$ of graded left $\Gamma^*(n)$-modules, adjoint to \eqref{pairing 0439}, is in fact an isomorphism of graded left $\Gamma^*$-modules,
\item $\sigma_{n,n+1}\circ\sigma_{n+1} = \sigma_n$ for all $n$, 
\item and the resulting graded left $\Gamma^*$-module homomorphism $\Gamma^* \rightarrow \lim_{n\rightarrow\infty}\Gamma^*(n)$ is an isomorphism.
\end{enumerate}
\end{itemize}
\end{definition}

\begin{definition}[Mitchell coalgebras and orientations]\leavevmode
\begin{itemize}
\item A {\em Mitchell coalgebra} is a graded coalgebra which admits a Mitchell decomposition. 
\item Given a Mitchell coalgebra $\Gamma$, by a {\em $\mathcal{P}$-sequence} we mean a sequence $\Gamma^*(0)\subseteq\Gamma^*(1)\subseteq\dots$ of graded $A$-subalgebras of $\Gamma^*$ as in the definition of a Mitchell decomposition. 
\item Given a Mitchell coalgebra and a choice of $\mathcal{P}$-sequence, by an {\em orientation sequence} we mean a sequence $\omega_0,\omega_1, \omega_2, \dots$ of elements as in the definition of a Mitchell decomposition.
\end{itemize}
\end{definition}

\begin{definition}[The ideal set $\Mit$ of a Mitchell decomposition]
 Given a Mitchell coalgebra $\Gamma$ equipped with a choice of Mitchell decomposition with orientation sequence $\omega_0,\omega_1, \omega_2, \dots$, let $\Mit$ be the set of all intersections of finite collections of homogeneous left ideals of $\Gamma^*$ of the form $\ann_{\ell}(\omega_n)$. That is, a homogeneous left ideal $I$ of $\Gamma^*$ is a member of $\Mit$ if and only if there exists a finite set $N$ of nonnegative integers such that $I = \underset{n\in N}{\bigcap}\ann_{\ell}(\omega_n)$.
\end{definition}

The main theorem of \cite{MR793186}, expressed in the language just introduced, is that, for each prime $p$, the $p$-primary dual Steenrod algebra is a Mitchell coalgebra.

\begin{lemma}\label{mitchell maps are surj}
Let $\Gamma$ be a graded coalgebra over a field $A$. Suppose $\Gamma$ is equipped with a choice of Mitchell decomposition. Then, for each nonnegative integer $n$, the graded left $\Gamma^*(n)$-module map $\sigma_{n,n+1}: \Gamma^*(n+1)\rightarrow \Gamma^*(n)$ is surjective.
\end{lemma}
\begin{proof}
The map $\sigma_n: \Gamma^* \rightarrow \Gamma^*(n)$ is a split $\Gamma^*(n)$-module epimorphism, and it factors as the composite $\sigma_{n,n+1}\circ \sigma_{n+1}$. Hence $\sigma_{n,n+1}$ is also surjective.
\end{proof}

Recall from Example \ref{examples of ideal sets 4} that $\grad$ is the filtered ideal set $\{ I_1, I_2, I_3, \dots\}$ in $\Gamma^*$, where $I_n$ is the left ideal (equivalently, two-sided ideal) in $\Gamma^*$ generated by all homogeneous elements of degree $\geq n$.
\begin{prop}\label{grad dist mit}
Suppose $A$ is a field and $\Gamma$ is a graded $A$-coalgebra concentrated in nonpositive degrees. Then we have $\dist\leq \grad$ in the preorder of filtered ideal sets in $\Gamma^*$. 

If $\Gamma$ is a finite-type Mitchell coalgebra, then we also have $\grad \leq\Mit\leq \dist$, and consequently $\grad$ and $\dist$ are equivalent in the preorder of filtered ideal sets in $\Gamma^*$. 
\end{prop}
\begin{proof}\leavevmode\begin{itemize}
\item 
Recall from Definition \ref{def of distinguished} that a homogeneous left ideal $I$ of $\Gamma^*$ is said to be {\em strongly distinguished} if the rational $\Gamma^*$-module $\iota(\Gamma)$ contains a graded $\Gamma^*$-submodule isomorphic to a suspension of $\Gamma^*/\Gamma^*I$.
Given a strongly distinguished homogeneous left ideal $I$ of $\Gamma^*$, let $\gamma\in \iota(\Gamma)$ be a homogeneous element whose left annihilator ideal $\ann_{\ell}(\gamma)$ is $I$. Let $n$ be the degree of the homogeneous element $\gamma$. Since $\Gamma$ is coconnective, $n$ must be nonpositive, so every element of $\Gamma^*$ of degree $>n$ annihilates $\gamma$, that is, $I_{n+1}\subseteq I$. 

Now by the definition of a distinguished ideal (Definition \ref{def of distinguished}), every element $J$ of $\dist$ contains an intersection $\cap_{j=1}^m J_j$ of a finite set $J_1, \dots ,J_m$ of strongly distinguished homogeneous left ideals of $\Gamma^*$. For each $j=1, \dots,m$, choose a nonnegative integer $f(j)$ such that $I_{f(j)}\subseteq J_j$. Then we have $J\supseteq \cap_{j=1}^m J_j\supseteq \cap_{j=1}^m I_{f(j)} = I_{\max\{ f(1), \dots,f(m)\}}$, so every element $J$ of $\dist$ contains an element $I_{\max\{ f(1), \dots,f(m)\}}$ of $\grad$.
So $\dist\leq \grad$ in the preorder of filtered ideal sets in $\Gamma^*$.
\item
If $\Gamma$ is a finite-type Mitchell coalgebra and $\omega_0,\omega_1,\omega_2,\dots$ is an orientation sequence for $\Gamma$, then by applying the dual $\sigma_n^*: \Gamma(n)\rightarrow \Gamma$ of the left $\Gamma^*$-module map $\sigma_n: \Gamma^*\rightarrow \Gamma^*(n)$ to $\omega_n\in \Gamma(n)$, we get a sequence of elements $\sigma_0^*(\omega_0),\sigma_1^*(\omega_1),\sigma_2^*(\omega_2),\dots$ of $\Gamma$. Since $\sigma_n$ is surjective, its dual $\sigma_n^*$ is injective, so $\ann_{\ell}(\sigma_n^*(\omega_n)) = \ann_{\ell}(\omega_n)$ for each $n$. Consequently each of the left ideals $\ann_{\ell}(\omega_n)$ of $\Gamma^*$ is strongly distinguished, so the ideal set $\Mit$ is contained in the ideal set $\dist$, and consequently $\Mit\leq \dist$.

To show that $\grad\leq \Mit$, choose some nonnegative integer $n$. We need to show that there exists some member $I$ of $\Mit$ such that $I_n\supseteq I$. That is, we need to find a member $I$ of $\Mit$ which has no nonzero elements of degree $<n$. At this point, we know that $\Gamma^*$ is finite-type, that the maps 
\begin{equation}\label{seq 34094f} \dots \rightarrow \Gamma^*(2) \rightarrow \Gamma^*(1) \rightarrow \Gamma^*(0)\end{equation}
are all grading-preserving and surjective (as a consequence of Lemma \ref{mitchell maps are surj}), and the map $\Gamma^*\rightarrow \lim_m\Gamma^*(m)$ is an isomorphism. Consequently, although the sequence \eqref{seq 34094f} is not necessarily eventually constant, in any single given degree it {\em is} eventually constant. 

In particular, there exists some positive integer $q_n$ such that the map $\Gamma^*\rightarrow \Gamma^*(m)$ is an isomorphism in degrees $\leq n$ for all $m\geq q_n$. By self-duality of $\Gamma^*(m)$, the left annihilator of $\omega_m$ contains no nonzero elements of $\Gamma^*(m)$ itself, so if $m\geq q_n$, $\ann_{\ell}(\omega_m)$ contains no nonzero elements of degree $<n$, as desired.
\end{itemize}
\end{proof}

To a filtered ideal set $S$ in a graded ring $R$, we have the associated local-cohomology-like functor $h_0^S$ given by $\underset{I\in S}{\colim}\ \underline{\hom}_R\left( R/RI,-\right) : \gr\Mod(R) \rightarrow \gr\Ab$. The functor $h_0^S$ comes equipped with a natural transformation $h_0^S \hookrightarrow F$, where $F$ is the forgetful functor $F: \gr\Mod(R)\rightarrow\gr\Ab$. 

Sometimes we are fortunate, and the image of the natural injection of abelian groups $h_0^S(M)\hookrightarrow M$ is actually an $R$-submodule of $M$. That is, when we are lucky, the set of $S$-torsion elements of $M$ is closed under left $R$-scalar multiplication, so that $h^0_S(M) \hookrightarrow H^0_S(M)$ is an isomorphism for all graded $R$-modules $M$. We will call the ideal set $S$ {\em closed} when this condition is satisfied. 

If $R$ is commutative, then every filtered ideal set in $R$ is closed. On the other hand, if $R$ is noncommutative, then $R$ may have some nonclosed filtered ideal sets. One example was given in Remark \ref{remark on ideal sets 2}. 

\begin{lemma}\label{closure closed under equivalence}
Let $R$ be a graded ring.
If two filtered ideal sets $S,S^{\prime}$ in $R$ are equivalent, and if $S$ is closed, then $S^{\prime}$ is also closed.
\end{lemma}
\begin{proof}
Easy consequence of the definitions.
\end{proof}

\begin{prop}\label{grad is closed}
Suppose $R$ is a connected graded algebra over a field. Then the filtered ideal set $\grad$ is closed.
\end{prop}
\begin{proof}
Elementary.
\end{proof}

\begin{theorem}\label{main thm on mitchell coalgebras}
Suppose $A$ is a field and $\Gamma$ is a finite-type Mitchell coalgebra over $A$. Then the following statements are each true:
\begin{enumerate}
\item The filtered ideal sets $\dist$ and $\grad$ in $\Gamma^*$ are equivalent. In particular, $\dist$ is closed.
\item The functor $h^0_{\dist}: \gr\Mod(R)\rightarrow\gr\Ab$ is equivalent to the composite of $H^0_{\dist}: \gr\Mod(R)\rightarrow\gr\Mod(R)$ with the forgetful functor $\gr\Mod(R)\rightarrow\gr\Ab$.
\item For every integer $n$, the distinguished local cohomology $H^n_{\dist}(M)$ is naturally isomorphic, as an abelian group, to $\underset{I\in \dist(\Gamma)}{\colim}\Ext_{\Gamma^*}\left( \Gamma^*/I,M\right)$.
\item For every integer $n$, the distinguished local cohomology $H^n_{\dist}(M)$ is naturally isomorphic, as an abelian group, to $\underset{j\rightarrow\infty}{\colim}\Ext_{\Gamma^*}\left( \Gamma^*/I_j,M\right)$, where $I_j$ is as in the definition of $\grad$, i.e., $I_j$ is the left ideal of $\Gamma^*$ generated by all homogeneous elements of degree $\geq j$.
\end{enumerate}
\end{theorem}
\begin{proof}\leavevmode
\begin{enumerate}
\item Immediate from Proposition \ref{grad dist mit}.
\item 
By Proposition \ref{grad is closed}, $\grad$ is closed.
By the previous part of this theorem, $\dist$ and $\grad$ are equivalent. By Lemma \ref{closure closed under equivalence}, $\dist$ is consequently also closed, so $h^0_{\dist}\cong H^0_{\dist}$.
\item Since $\dist$ is filtered, the colimit $\underset{I\in \dist}{\colim}$ is exact, so we have 
\begin{align*}
 H^n_{\dist}(M) 
  &\cong R^nH^0_{\dist}(M) \\
  &\cong R^nh^0_{\dist}(M) \\
  &\cong R^n\left(\underset{I\in \dist}{\colim}\ \underline{\hom}_{\Gamma^*}\left( \Gamma^*/\Gamma^*I,-\right)\right)(M) \\
  &\cong \underset{I\in \dist}{\colim} \Ext_R^n\left( \Gamma^*/\Gamma^*I,M\right).
\end{align*}
\item Immediate from the preceding part of this theorem together with the equivalence of the filtered ideal sets $\grad$ and $\dist$.
\end{enumerate}
\end{proof}

See Remark \ref{remark on h0 and H0} for some discussion of why Theorem \ref{main thm on mitchell coalgebras} matters.

\begin{corollary}\label{main cor 10}
Let $p$ be a prime number, and let $\Gamma^*$ be the mod $p$ Steenrod algebra. Let $M$ be a graded $\Gamma^*$-module. Then, for all integers $n$, the distinguished local cohomology group $H^n_{\dist}(M)$ is isomorphic to the graded local cohomology group $\underset{j\rightarrow\infty}{\colim} \Ext_{\Gamma^*}^n\left( \Gamma^*/I_j,M\right)$, where $I_j$ is the ideal of the Steenrod algebra generated by all homogeneous elements of degree $\geq j$. 
\end{corollary}

\begin{corollary}\label{derived products are local cohomology}
Let $p$ be a prime, let $\Gamma$ be the mod $p$ dual Steenrod algebra, and let $\{ M_i: i\in I\}$ be a set of bounded-above\footnote{To be clear, the comodules $M_i$ do not need to be {\em uniformly} bounded above.} graded $\Gamma$-comodules. Then the $n$th derived functor $R^n\prod^{\Gamma}_{i\in I} M_i$ of product in the category of graded $\Gamma$-comodules is isomorphic, as an abelian group, to the graded local cohomology $\underset{j\rightarrow\infty}{\colim} \Ext_{\Gamma^*}^n\left( \Gamma^*/I_j,\prod_i M_i\right)$. Here $I_j$ is as in Corollary \ref{main cor 10}, and $\prod_i M_i$ is the Cartesian product of the graded $\Gamma^*$-modules $M_i$ with the adjoint action of $\Gamma^*$.
\end{corollary}

As far as the author knows, the original reference on graded local cohomology is \cite{MR494707}, which only considered commutative Noetherian rings. In the noncommutative case, \cite{MR1428799} seems to be the first reference, although it still assumes a Noetherian condition on the graded noncommutative ring. Unfortunately we have not found any results in the literature on local cohomology which appear to be useful in the case where $R$ is a Steenrod algebra, since as far as we have been able to determine, none of the existing literature on local cohomology treats rings which are both noncommutative and non-Noetherian. With the topological applications provided by Corollary \ref{derived products are local cohomology} and the Sadofsky and Hovey-Sadofsky spectral sequences, there is good reason to generalize the existing computational tools for local cohomology to the case of graded noncommutative non-Noetherian rings. That task lies beyond the scope of this paper, though.

\section{Bounds on distinguished local-cohomological dimension.}
\label{Bounds on distinguished...}


The most fundamental and well-known vanishing theorem in classical local cohomology is this: if $I$ is an ideal in a Noetherian commutative ring, then the classical local cohomology groups $H^n_I(M)$ vanish for all $R$-modules $M$ whenever $n$ is greater than the least number of generators for $I$. One wants a generalization of this result which applies to distinguished local cohomology. However, $h^*_{\dist}$ is given by a colimit of $\Ext$ groups associated to a set $S$ of ideals which is not generally the sequence of powers of any single ideal $I$, so it is not clear whether one ought to expect $h^n_{\dist}(M)$ to vanish for all $n$ larger than some particular integer. It is even less clear what to expect from $H^n_{\dist}$ except when the relevant coalgebra $\Gamma$ is either co-commutative, or finite-type and Mitchell, so that $H^n_{\dist}$ agrees with $h^n_{\dist}$.

The main result of this section is Theorem \ref{main thm 4}, which establishes that, for a certain class of graded coalgebras $\Gamma$ (including those whose dual is a $\mathcal{P}$-algebra, so in particular, including the dual Steenrod algebras), there is no such vanishing theorem for distinguished local cohomology. As a consequence we get that the graded comodule category over such a coalgebra fails to be $AB4^*\mhyphen (n)$ for any $n$ whatsoever.

\begin{theorem}\label{main thm 4}
Suppose that $A$ is a field, and that $\Gamma^*$ is finite-type and connected.
Suppose furthermore that every bounded-below free graded $\Gamma^*$-module is injective in $\gr\Mod(\Gamma^*)$.
Then the following conditions are equivalent:
\begin{enumerate}
\item The category of graded $\Gamma^*$-modules has distinguished local cohomological dimension zero. That is, $H^m_{\dist}(M)$ vanishes for all positive $m$ and all graded $\Gamma^*$-modules $M$.
\item The category of graded $\Gamma^*$-modules has finite distinguished local cohomological dimension. That is, there exists some integer $n$ such that $H^m_{\dist}(M)$ vanishes for all $m>n$ and all graded $\Gamma^*$-modules $M$.
\item The category of bounded-above graded $\Gamma$-comodules satisfies axiom $AB4^*\mhyphen (n)$ for some nonnegative integer $n$. 
That is, there exists some integer $n$ such that $R^m\prod^{\Gamma}_i\left( \left\{ M_i\right\}\right)$ vanishes for all $m>n$, all countable sets $I$, and all sets $\{ M_i:i\in I\}$ of bounded-above graded $\Gamma$-comodules.
\item The category of bounded-above graded $\Gamma$-comodules satisfies Grothendieck's property $AB4^*$. That is, countable products are exact on bounded-above graded $\Gamma$-comodules.
\end{enumerate}
\end{theorem}
\begin{proof}\leavevmode
\begin{description}
\item[(1) implies (2)] Immediate.
\item[(2) implies (1)] Let $M$ be a graded $\Gamma^*$-module. By Theorem \ref{structure thm}, $M$ has the same distinguished local cohomology in positive degrees as the graded $\Gamma^*$-module $\conn_0(M)$ of $M$ generated by all homogeneous elements of nonnegative degree. Since $\conn_0(M)$ is trivial in negative degrees, there exists a connective free graded $\Gamma^*$-module $F$ and an epimorphism $\epsilon: F \rightarrow \conn_0(M)$. Since $F$ is injective, 
we have $H^{i+1}_{\dist}(\ker \epsilon)\cong H^{i}_{\dist}(\conn_0(M))\cong H^{i}_{\dist}(M)$ for all $i\geq 1$. So if $H^i_{\dist}$ is nonzero on some graded $\Gamma^*$-module for some positive $i$, then $H^{i+1}_{\dist}$ is also nonzero on some graded $\Gamma^*$-module. Consequently, under the stated hypotheses, the only way to have a finite cohomological bound on distinguished local cohomology is for that bound to be zero.
\item[(2) implies (3)] Special case of Theorem \ref{main thm 1}.
\item[(1) implies (4)] Special case of Theorem \ref{main thm 1}.
\item[(4) implies (3)] Immediate, since $AB4^*$ is the same condition as $AB4^*\mhyphen (0)$.
\item[(3) implies (2)] Suppose that the category of bounded-above graded \linebreak $\Gamma$-comodules satisfies axiom $AB4^*\mhyphen (n)$ for some nonnegative integer $n$. Suppose that $M$ is a graded $\Gamma^*$-module. By Corollary \ref{modules are limits of comodules}, $M$ is isomorphic to the limit of a Mittag-Leffler sequence $\dots \rightarrow N_2 \rightarrow N_1 \rightarrow N_0$ of rational, bounded-above graded $\Gamma^*$-modules. Consequently we have a short exact sequence of graded $\Gamma^*$-modules
\begin{equation}\label{ses 30-204949} 0 \rightarrow M \rightarrow \prod_{i\in I} N_i \stackrel{\id - T}{\longrightarrow} \prod_{i\in I} N_i \rightarrow 0\end{equation} with each of the modules $N_0,N_1,N_2,\dots$ bounded-above.
Since each of the $N_i$ are bounded-above, Theorem \ref{main thm 1} gives us an isomorphism of abelian groups \begin{align*} H^*_{\dist}\left( \prod_{i\in I} N_i\right) &\cong R^*\prod^{\Gamma}_i\left( \{ N_i :i\in\mathbb{N}\}\right),\end{align*} so the $AB4^*\mhyphen (n)$ assumption on the comodule category gives us that \linebreak $H^j_{\dist}\left( \prod_{i\in I} N_i\right)$ vanishes for all $j>n$.
The long exact sequence obtained by applying $H^*_{\dist}$ to \eqref{ses 30-204949} consequently gives us that $H^j_{\dist}(M)$ vanishes for all $j>n+1$.
That is, the category of graded $\Gamma^*$-modules has distinguished local cohomological dimension at most $n+1$.
\end{description}
\end{proof}

\begin{corollary}\label{main cor 4a}
Suppose that $A$ is a field, and suppose that the $A$-algebra $\Gamma^*$ is a finite-type $\mathcal{P}$-algebra. Then {\em exactly one} of the two following statements is true:
\begin{enumerate}
\item The category of bounded-above graded $\Gamma$-comodules satisfies condition $AB4^*$. That is, products are exact in the category of graded $\Gamma$-comodules.
\item For each integer $n$, the category of bounded-above graded $\Gamma$-comodules fails to satisfy condition $AB4^*\mhyphen (n)$. That is, for each integer $n$, there exists a countably infinite set $\{ M_0, M_1, M_2, \dots\}$ of graded $\Gamma$-comodules such that $R^m\prod^{\Gamma}_i \left( \{ M_i\}\right)$ is nonzero for some $m>n$.
\end{enumerate}
\end{corollary}
\begin{proof}
By Theorem \ref{thm from margolis 2} from the appendix on Margolis' results on $\mathcal{P}$-algebras, every bounded-below free graded $\Gamma^*$-module is injective. Consequently the hypotheses of Theorem \ref{main thm 4} are satisfied.
\end{proof}

\begin{corollary}\label{no ab4n for dual steenrod alg}
Let $p$ be a prime number, and let $\Gamma$ denote the mod $p$ dual Steenrod algebra. Then, for each integer $n$, the category of bounded-above graded $\Gamma$-comodules fails to satisfy condition $AB4^*\mhyphen (n)$.
That is, there exists a countably infinite set $M_0, M_1, M_2, \dots$ of bounded-above graded $\Gamma$-comodules such that the derived product $R^m\prod^{\Gamma}_i\left(\left\{ M_i: i\in \mathbb{N}\right\}\right)$ is nonzero for some $m>n$.
\end{corollary}
\begin{proof}
It is already well-known (but we still give some explanation in the rest of this proof) that the category of graded comodules over the dual Steenrod algebra does not have exact products, so the claim is a corollary of Corollary \ref{main cor 4a}.

If the category of bounded-above graded comodules over the dual Steenrod algebra had exact products, this would imply that, for any countable set $\{ X_i\}$ of bounded-below $H\mathbb{F}_p$-nilpotently complete spectra, the Hovey-Sadofsky spectral sequence 
\[ E_2^{*,*} \cong R^*\prod_{i\in I}^{\Gamma}\left( \{ H_*(X_i;\mathbb{F}_p)\}\right) \Rightarrow H_*\left( \prod_{i\in I} X_i;\mathbb{F}_p\right)\] collapses on to the $R^0\prod^{\Gamma}$-line at the $E_2$-page. One can consult Sadofsky's unpublished preprint \cite{sadofsky2001homology} for the original construction of this spectral sequence (which is given there more generally for sequential limits, not just countable products), Hovey's paper \cite{MR2337861} for a published account (which, however, discusses convergence only when $H_*(-;\mathbb{F}_p)$ is replaced with a Morava $E$-theory), and the appendix of \cite{MR4080481} for a published account which discusses convergence of the spectral sequence in the case at hand, i.e., the case of mod $p$ homology. Collapse of the spectral sequence for all $\{ X_i: i\in\mathbb{N}\}$ would imply that mod $p$ homology commutes with all countable products of bounded-below $H\mathbb{F}_p$-nilpotently complete spectra, and consequently that mod $p$ homology commutes with all sequential homotopy limits of bounded-below $H\mathbb{F}_p$-nilpotently complete spectra, which is well-known to be untrue without some additional hypothesis (e.g. that the spectra are {\em uniformly} bounded below). 
\end{proof}

\section{Derived functors of sequential limit in comodule categories.}
\label{Derived functors of seq...}

Let $\Gamma$ be a graded coalgebra over a field $A$. We continue to write $\iota: \gr\Comod(\Gamma) \rightarrow \gr\Mod(\Gamma^*)$ for the inclusion of comodules into modules via the adjoint $\Gamma^*$-action, and we continue to write $\tr: \gr\Mod(\Gamma^*) \rightarrow \gr\Comod(\Gamma)$ for the right adjoint of $\iota$.


Let $\mathcal{D}$ be a small category. 
For any functor $\mathcal{F}:\mathcal{D}\rightarrow\gr\Comod(\Gamma)$, the Bousfield-Kan construction of $\iota\circ\mathcal{F}$ is the cosimplicial graded $\Gamma^*$-module
\begin{equation}\label{bk construction} \xymatrix{ 
\underset{x\in N\mathcal{D}_0}{\prod} \iota\mathcal{F}(\cod x) \ar@<1ex>[r] \ar@<-1ex>[r] & 
 \underset{x\in N\mathcal{D}_1}{\prod} \iota\mathcal{F}(\cod x)  \ar[l]\ar@<2ex>[r]\ar[r]\ar@<-2ex>[r] & 
 \underset{x\in N\mathcal{D}_2}{\prod} \iota\mathcal{F}(\cod x)  \ar@<1ex>[l] \ar@<-1ex>[l] \ar@<3ex>[r] \ar@<1ex>[r] \ar@<-1ex>[r] \ar@<-3ex>[r] & \dots \ar@<2ex>[l]\ar[l]\ar@<-2ex>[l] , }\end{equation}
The notation is as follows: $N\mathcal{D}_{\bullet}$ is the nerve of $\mathcal{D}$, so for $i>0$, $N\mathcal{D}_{i}$ is the set of composable $i$-tuples of morphisms of $\mathcal{D}$, while $N\mathcal{D}_{0}$ is the set of objects of $\mathcal{D}$. The notation $\cod x$ is supposed to denote the final object appearing in a composable $i$-tuple of morphisms $x$; in the case $i=1$, this is simply the codomain of a morphism.
We write $C^{\bullet}(\iota\mathcal{F})$ for the Moore complex (i.e., alternating sum cochain complex) of \eqref{bk construction}, so that $H^i(C^{\bullet}(\iota\mathcal{F})) \cong (R^i\lim)(\iota\mathcal{F})$ by Proposition XI.6.2 of \cite{MR0365573}.

In Theorem \ref{derived sequential lim thm}, we write $\lim^{\Gamma}$ for a limit taken in the category of graded $\Gamma$-comodules.
\begin{theorem}\label{derived sequential lim thm}
Suppose that $A$ is a field and that $\Gamma$ is a finite-type $A$-coalgebra such that $\Gamma^*$ is connected. 
Let \begin{equation}\label{seq 430949598}\dots\rightarrow M_2\rightarrow M_1\rightarrow M_0\end{equation} be a sequence of bounded-above\footnote{Again, to be clear: these comodules do not need to be {\em uniformly} bounded above.} graded $\Gamma$-comodules such that $R^1\lim$ vanishes on the sequence of graded $\Gamma^*$-modules 
$\dots\rightarrow \iota M_2\rightarrow \iota M_1\rightarrow \iota M_0$.
Then there is an isomorphism of $\Gamma^*$-modules
\begin{align}\label{iso 09994} 
 \iota R^s\lim^{\Gamma}_iM_i &\cong H^s_{\dist}(\lim_i \iota M_i)\end{align}
for each $s$,
and consequently a long exact sequence
\begin{equation}\label{les 239094}\xymatrix{
 0 \ar[r] & 
  \iota \lim^{\Gamma}_i M_i \ar[r] & 
  H^0_{\dist}\left( \prod_i \iota M_i\right) \ar[r] &
  H^0_{\dist}\left( \prod_i \iota M_i\right) \ar`r_l[ll] `l[dll] [dll] \\
 & \iota R^1\lim^{\Gamma}_i M_i \ar[r] &
  H^1_{\dist}\left( \prod_i \iota M_i\right) \ar[r] &
  H^1_{\dist}\left( \prod_i \iota M_i\right) \ar`r_l[ll] `l[dll] [dll] \\
 & \iota R^2\lim^{\Gamma}_i M_i \ar[r] &
  H^2_{\dist}\left( \prod_i \iota M_i\right) \ar[r] &
  \dots .}\end{equation}
\end{theorem}
\begin{proof}
Consider the natural numbers $\mathbb{N}$ under their usual ordering, and let $\mathbb{N}^{\op}$ be the opposite category of the resulting partially-ordered set. Then the sequence \eqref{seq 430949598} is a functor $\mathcal{F}: \mathbb{N}^{\op}\rightarrow\gr\Comod(\Gamma)$.
Since $\tr$ is a right adjoint, it is left exact, so we have the hypercohomology spectral sequences
\begin{align*}
 {}^{\prime}E_1^{s,t} 
  & \cong R^t\tr(C^s\iota \mathcal{F}) \\
 {}^{\prime\prime}E_2^{s,t} 
  &\cong R^s\tr(H^t C^{\bullet}\iota\mathcal{F}) \\
  &\cong R^s\tr(R^t\lim \iota\mathcal{F}) ,
\end{align*}
each of which converges to the cohomology of the total complex of $\tr$ applied to $C^{\bullet}\iota\mathcal{F}$. 
Since $R^t\lim \iota\mathcal{F}$ is a derived functor of sequential limit in a category of modules, it is only capable of being nontrivial in degrees $t=0$ and $t=1$. Furthermore, $R^1\lim \iota\mathcal{F}$ vanishes by assumption. Hence the spectral sequence ${}^{\prime\prime}E_r^{*,*}$ collapses at the ${}^{\prime\prime}E_2$-page, with ${}^{\prime\prime}E_2^{s,0}\cong R^s\tr(\lim\iota\mathcal{F})$ and with ${}^{\prime\prime}E_2^{s,t}$ trivial for $t>0$.

Meanwhile, ${}^{\prime}E_1^{*,t}$ is precisely the Moore complex $C^{\bullet}\left((R^t\tr) \iota\mathcal{F}\right)$ of the Bousfield-Kan construction of $(R^t\tr) \circ \iota\circ \mathcal{F}$, since $\tr$ commutes with products and hence $R^t\tr$ commutes with products. So 
\begin{align}
\nonumber {}^{\prime}E_2^{s,t} &\cong R^s\lim^{\Gamma} (R^t\tr (\iota\mathcal{F}))\\ 
\label{iso 534084}&\cong R^s\lim^{\Gamma}_i (R^t\tr (\iota M_i)),\end{align}
with \eqref{iso 534084} due to exactness of $\iota$.
Now ${}^{\prime}E_2^{s,t}$ is trivial for $t>0$, since $\iota$ is exact and faithful and since $\iota R^t\tr (\iota M_i) \cong H^t_{\dist}(\iota M_i) \cong 0$ for $t>0$, as distinguished local cohomology vanishes on bounded-above modules, by Theorem \ref{main thm 1}. 
Hence ${}^{\prime}E_2^{s,0} \cong R^s\lim^{\Gamma}( \tr\iota\mathcal{F}) \cong R^s\lim^{\Gamma} \mathcal{F}$, and ${}^{\prime}E_2^{s,t}$ is trivial for $t\neq 0$.

Since both spectral sequences converge to the cohomology of the same total complex, we have $R^n\lim^{\Gamma} \mathcal{F} \cong R^n\tr(\lim\iota\mathcal{F})$ for all $n$, and on $\iota$ and using Theorem \ref{main thm 1}, we have the desired isomorphism
\begin{align}\label{iso 40059454} \iota R^n\lim_i^{\Gamma}(M_i) &\cong H^n_{\dist}\left(\lim_i\iota M_i\right).\end{align}

Finally, long exact sequence \eqref{les 239094} arises from applying $H^n_{\dist}$ to the Milnor-type short exact sequence of graded $\Gamma^*$-modules
\[ 0 \rightarrow \lim_i \iota M_i \rightarrow \prod_i \iota M_i \rightarrow \prod_i \iota M_i \rightarrow 0 \]
and using isomorphism \eqref{iso 40059454}.
\end{proof}

\begin{corollary}
Suppose that $A$ is a field and that $\Gamma$ is a finite-type Mitchell coalgebra over $A$. Let $\dots\rightarrow M_2\rightarrow M_1\rightarrow M_0$ be a sequence of bounded-above graded $\Gamma$-comodules such that $R^1\lim$ vanishes on the sequence of graded $\Gamma^*$-modules 
$\dots\rightarrow \iota M_2\rightarrow \iota M_1\rightarrow \iota M_0$.
Then we have an isomorphism of abelian groups
\[ R^*\lim^{\Gamma}_i M_i  \cong \underset{j\rightarrow\infty}{\colim}\Ext^*_{\Gamma^*}\left(\Gamma^*/I_j,\lim_i \iota M_i\right)\]
where $I_j$ is the ideal of $\Gamma^*$ generated by all homogeneous elements of degree $\geq j$.
\end{corollary}
\begin{proof}
Consequence of Theorems \ref{derived sequential lim thm} and \ref{main thm on mitchell coalgebras}.
\end{proof}

\begin{corollary}\label{steenrod alg seq cor}
Suppose that $p$ is a prime number. Write $\Gamma$ for the mod $p$ dual Steenrod algebra. Let $\dots\rightarrow M_2\rightarrow M_1\rightarrow M_0$ be a sequence of bounded-above graded $\Gamma$-comodules such that $R^1\lim$ vanishes on the sequence of graded $\Gamma^*$-modules 
$\dots\rightarrow \iota M_2\rightarrow \iota M_1\rightarrow \iota M_0$.
Then we have an isomorphism of abelian groups
between $R^*\lim^{\Gamma}_i M_i$ and the graded local cohomology $\underset{j\rightarrow\infty}{\colim}\Ext^*_{\Gamma^*}\left(\Gamma^*/I_j,\lim_i \iota M_i\right)$ of the Steenrod algebra $\Gamma^*$.
\end{corollary}

\appendix
\section{Review of Margolis' basic results on $\mathcal{P}$-algebras.}
\label{Review of Margolis...}

Margolis' fundamental results on graded module theory over $\mathcal{P}$-algebras are Theorem 5 from section 13.2 and Theorem 12 from section 13.3 of Margolis's book \cite{MR738973}. We state those two results:
\begin{theorem}{\bf (Margolis.)} \label{thm from margolis} Let $B$ be a $\mathcal{P}$-algebra, and let $B(0) \subsetneq B(1) \subsetneq \dots$ be a sequence of Poincar\'{e} subalgebras of $B$, as in Definition \ref{def of p-alg}. Let $M$ be a graded left $B$-module. Then the following are equivalent:
\begin{enumerate}
\item The projective dimension of $M$ is $\leq 1$.
\item $M$ is flat.
\item The injective dimension of $M$ is $\leq 1$.
\item $M$ is free over $B(n)$ for each $n$.
\end{enumerate}
Furthermore, if $M$ does not satisfy these (equivalent) conditions, then the projective dimension, weak dimension, and injective dimension of $M$ are each infinite.
\end{theorem}
\begin{theorem}{\bf (Margolis.)} \label{thm from margolis 2} Let $B$ be a $\mathcal{P}$-algebra, and let $B(0) \subsetneq B(1) \subsetneq \dots$ be a sequence of Poincar\'{e} subalgebras of $B$, as in Definition \ref{def of p-alg}. Let $M$ be a {\em bounded-below} graded left $B$-module. Then the following are equivalent:
\begin{enumerate}
\item $M$ is free.
\item $M$ is projective.
\item $M$ is flat.
\item $M$ is injective.
\item $M$ is free over $B(n)$ for each $n$.
\end{enumerate}
Furthermore, if $M$ does not satisfy these (equivalent) conditions, then the projective dimension, weak dimension, and injective dimension of $M$ are each infinite.
\end{theorem}


\begin{thebibliography}{10}

\bibitem{MR1324104}
J.~F. Adams.
\newblock {\em Stable homotopy and generalised homology}.
\newblock Chicago Lectures in Mathematics. University of Chicago Press,
  Chicago, IL, 1995.
\newblock Reprint of the 1974 original.

\bibitem{MR4060488}
Tobias Barthel, Drew Heard, and Gabriel Valenzuela.
\newblock Derived completion for comodules.
\newblock {\em Manuscripta Math.}, 161(3-4):409--438, 2020.

\bibitem{barthel_pstragowski_2023}
Tobias Barthel and Piotr Pstrągowski.
\newblock Morava {K}-theory and filtrations by powers.
\newblock {\em Journal of the Institute of Mathematics of Jussieu}, page
  1–77, 2023.

\bibitem{MR4094969}
Mark Behrens and Charles Rezk.
\newblock The {B}ousfield-{K}uhn functor and topological {A}ndr\'{e}-{Q}uillen
  cohomology.
\newblock {\em Invent. Math.}, 220(3):949--1022, 2020.

\bibitem{MR1780016}
Apostolos Beligiannis.
\newblock Cleft extensions of abelian categories and applications to ring
  theory.
\newblock {\em Comm. Algebra}, 28(10):4503--4546, 2000.

\bibitem{MR1313497}
Francis Borceux.
\newblock {\em Handbook of categorical algebra. 2}, volume~51 of {\em
  Encyclopedia of Mathematics and its Applications}.
\newblock Cambridge University Press, Cambridge, 1994.
\newblock Categories and structures.

\bibitem{MR0365573}
A.~K. Bousfield and D.~M. Kan.
\newblock {\em Homotopy limits, completions and localizations}.
\newblock Lecture Notes in Mathematics, Vol. 304. Springer-Verlag, Berlin,
  1972.

\bibitem{MR3014449}
M.~P. Brodmann and R.~Y. Sharp.
\newblock {\em Local cohomology}, volume 136 of {\em Cambridge Studies in
  Advanced Mathematics}.
\newblock Cambridge University Press, Cambridge, second edition, 2013.
\newblock An algebraic introduction with geometric applications.

\bibitem{brunerprimer}
R.~R. Bruner.
\newblock {\em An Adams spectral sequence primer}.
\newblock in progress; draft versions available from Bruner's website, 2019.

\bibitem{MR2012570}
Tomasz Brzezinski and Robert Wisbauer.
\newblock {\em Corings and comodules}, volume 309 of {\em London Mathematical
  Society Lecture Note Series}.
\newblock Cambridge University Press, Cambridge, 2003.

\bibitem{MR191935}
Spencer~E. Dickson.
\newblock A torsion theory for {A}belian categories.
\newblock {\em Trans. Amer. Math. Soc.}, 121:223--235, 1966.

\bibitem{MR880019}
Jonathan~S. Golan.
\newblock {\em Torsion theories}, volume~29 of {\em Pitman Monographs and
  Surveys in Pure and Applied Mathematics}.
\newblock Longman Scientific \& Technical, Harlow; John Wiley \& Sons, Inc.,
  New York, 1986.

\bibitem{MR494707}
Shiro Goto and Keiichi Watanabe.
\newblock On graded rings. {I}.
\newblock {\em J. Math. Soc. Japan}, 30(2):179--213, 1978.

\bibitem{MR0102537}
Alexander Grothendieck.
\newblock Sur quelques points d'alg\`ebre homologique.
\newblock {\em T\^ohoku Math. J. (2)}, 9:119--221, 1957.

\bibitem{MR0222093}
Robin Hartshorne.
\newblock {\em Residues and duality}.
\newblock Lecture notes of a seminar on the work of A. Grothendieck, given at
  Harvard 1963/64. With an appendix by P. Deligne. Lecture Notes in
  Mathematics, No. 20. Springer-Verlag, Berlin, 1966.

\bibitem{hogadixu2009}
Amit Hogadi and Chenyang Xu.
\newblock Products, {H}omotopy {L}imits and {A}pplications.
\newblock {\em arXiv preprint arXiv:0902.4016}, 2009.

\bibitem{MR2337861}
Mark Hovey.
\newblock The generalized homology of products.
\newblock {\em Glasg. Math. J.}, 49(1):1--10, 2007.

\bibitem{MR1428799}
Peter J{\o}rgensen.
\newblock Local cohomology for non-commutative graded algebras.
\newblock {\em Comm. Algebra}, 25(2):575--591, 1997.

\bibitem{lam2013first}
Tsit-Yuen Lam.
\newblock {\em A first course in noncommutative rings}, volume 131.
\newblock Springer Science \& Business Media, 2013.

\bibitem{MR2238922}
Wendy Lowen and Michel Van~den Bergh.
\newblock Deformation theory of abelian categories.
\newblock {\em Trans. Amer. Math. Soc.}, 358(12):5441--5483, 2006.

\bibitem{MR738973}
H.~R. Margolis.
\newblock {\em Spectra and the {S}teenrod algebra}, volume~29 of {\em
  North-Holland Mathematical Library}.
\newblock North-Holland Publishing Co., Amsterdam, 1983.

\bibitem{MR1213784}
Nikolaos Marmaridis.
\newblock On extensions of abelian categories with applications to ring theory.
\newblock {\em J. Algebra}, 156(1):50--64, 1993.

\bibitem{MR793186}
Stephen~A. Mitchell.
\newblock Finite complexes with {$A(n)$}-free cohomology.
\newblock {\em Topology}, 24(2):227--246, 1985.

\bibitem{MR335572}
John~C. Moore and Franklin~P. Peterson.
\newblock Nearly {F}robenius algebras, {P}oincar\'{e} algebras and their
  modules.
\newblock {\em J. Pure Appl. Algebra}, 3:83--93, 1973.

\bibitem{MR551625}
Constantin N\u{a}st\u{a}sescu and F.~Van~Oystaeyen.
\newblock {\em Graded and filtered rings and modules}, volume 758 of {\em
  Lecture Notes in Mathematics}.
\newblock Springer, Berlin, 1979.

\bibitem{MR4080481}
Eric Peterson.
\newblock Coalgebraic formal curve spectra and spectral jet spaces.
\newblock {\em Geom. Topol.}, 24(1):1--47, 2020.

\bibitem{MR860042}
Douglas~C. Ravenel.
\newblock {\em Complex cobordism and stable homotopy groups of spheres}, volume
  121 of {\em Pure and Applied Mathematics}.
\newblock Academic Press Inc., Orlando, FL, 1986.

\bibitem{MR132091}
Jan-Erik Roos.
\newblock Sur les foncteurs d\'{e}riv\'{e}s de {$\underleftarrow\lim$}.
  {A}pplications.
\newblock {\em C. R. Acad. Sci. Paris}, 252:3702--3704, 1961.

\bibitem{MR2197371}
Jan-Erik Roos.
\newblock Derived functors of inverse limits revisited.
\newblock {\em J. London Math. Soc. (2)}, 73(1):65--83, 2006.

\bibitem{sadofsky2001homology}
Hal Sadofsky.
\newblock The homology of inverse limits of spectra.
\newblock {\em unpublished preprint}, 2001.

\bibitem{selfinjectivitypreprint}
A.~Salch.
\newblock The {S}teenrod algebra is self-injective, and the {S}teenrod algebra
  is not self-injective.
\newblock 2023.
\newblock arXiv preprint.

\bibitem{salchenoughproj}
Andrew Salch.
\newblock Graded comodule categories with enough projectives.
\newblock {\em Glasgow Mathematical Journal}, page 1–15, 2022.

\bibitem{MR0389953}
Bo~Stenstr\"{o}m.
\newblock {\em Rings of quotients}.
\newblock Die Grundlehren der mathematischen Wissenschaften, Band 217.
  Springer-Verlag, New York-Heidelberg, 1975.
\newblock An introduction to methods of ring theory.

\bibitem{MR3565424}
Simone Virili.
\newblock On the exactness of products in the localization of {$\rm (Ab.4^*)$}
  {G}rothendieck categories.
\newblock {\em J. Algebra}, 470:37--67, 2017.

\end{thebibliography}

\def\cprime{$'$} \def\cprime{$'$} \def\cprime{$'$} \def\cprime{$'$}

\end{document}